\documentclass[11pt]{article}
\usepackage[latin1]{inputenc}
\usepackage[margin=0.8in]{geometry}
\usepackage{amsmath,amssymb,amstext}
\usepackage{amsthm}
\usepackage{cleveref}
\usepackage{authblk}
\usepackage{enumerate}
\usepackage{enumitem}
\usepackage{dirtytalk}
\usepackage{tikz}
\usepackage{thmtools}
\usepackage{thm-restate}
\usepackage{multicol}
\usetikzlibrary{shapes}

\usepackage{tikz}
\usepackage{thmtools}
\usepackage{thm-restate}
\usetikzlibrary{shapes}

\newtheorem{thm}{Theorem}[section]
\newtheorem{prop}[thm]{Proposition}
\newtheorem{lemma}[thm]{Lemma}

\newtheorem{cor}[thm]{Corollary}
\newtheorem{claim}{Claim}%[]

\newtheorem{obs}{Observation}

\numberwithin{equation}{section}
\theoremstyle{definition} 
\newtheorem{definition}[thm]{Definition}

\newcommand{\F}{F}

\newcommand{\defc}{\textnormal{def}}

\definecolor{asparagus}{rgb}{0.53, 0.66, 0.42}

\definecolor{cerulean}{rgb}{0.0, 0.48, 0.65}
\definecolor{cornellred}{rgb}{0.7, 0.11, 0.11}
\definecolor{darklavender}{rgb}{0.45, 0.31, 0.59}
\definecolor{darkslateblue}{rgb}{0.28, 0.24, 0.55}
\definecolor{burntorange}{rgb}{0.8, 0.33, 0.0}
\date{}

\begin{document}

%TITLE, ETC

%\author{Luke Postle\thanks{$^{,\dagger}$We acknowledge the support of the Natural Sciences and Engineering Research Council of Canada (NSERC), $^*$[Discovery Grant No.  2019-04304], $^\dagger$[CGSD3 Grant No. 2020-547516] \\
%\hphantom{|} $ ^{*,\dagger}$Cette recherche a \'{e}t\'{e} financ\'{e}e par le Conseil de recherches en sciences naturelles et en g\'{e}nie du Canada (CRSNG), $^*$[Discovery Grant No.  2019-04304], $^\dagger$[CGSD3 Grant No. 2020-547516]} }
\author{Luke Postle\thanks{We acknowledge the support of the Natural Sciences and Engineering Research Council of Canada (NSERC) [Discovery Grant No.  2019-04304].\\
\hphantom{m} $^*$Cette recherche a \'{e}t\'{e} financ\'{e}e par le Conseil de recherches en sciences naturelles et en g\'{e}nie du Canada (CRSNG)[Discovery Grant No.  2019-04304].} }
\author{Evelyne Smith-Roberge$^\dagger$}
\affil{$^*$Dept.~of Combinatorics and Optimization, University of Waterloo \\ \texttt{lpostle@uwaterloo.ca}}
\affil{$^\dagger$School of Mathematics, Georgia Institute of Technology \\ \texttt{esmithroberge3@gatech.edu}}
\title{Hyperbolicity Theorems for Correspondence Colouring}
\date{\today}
\maketitle
\date{}

%------------------------------
%-------  INTRODUCTION --------
%------------------------------
\begin{abstract}
We generalize a framework of list colouring results to \emph{correspondence colouring}. Correspondence colouring is a generalization of list colouring wherein we localize the meaning of the colours available to each vertex. As pointed out by Dvo{\v{r}}{\'a}k and Postle, both of Thomassen's theorems on the 5-choosability of planar graphs and 3-choosability of planar graphs of girth at least five carry over to the correspondence colouring setting. In this paper, we show that the family of graphs that are critical for 5-correspondence colouring as well as the family of graphs of girth at least five that are critical for 3-correspondence colouring form \emph{hyperbolic families}. Analogous results for list colouring were shown by Postle and Thomas and by Dvo{\v{r}}{\'a}k and Kawarabayashi, respectively. Using results on hyperbolic families due to Postle and Thomas, we show further that this implies that locally planar graphs are 5-correspondence colourable; and, using results of Dvo{\v{r}}{\'a}k and Kawarabayashi, that there exist linear-time algorithms for the decidability of 5-correspondence colouring for embedded graphs. We show analogous results for 3-correspondence colouring graphs of girth at least five.\end{abstract}
\maketitle

%added this for a new section style
%\newsectionstyle
%======================================================================
\section{Introduction}\label{sec:intro}
%======================================================================

List colouring, a natural generalization of vertex colouring, was first introduced in the 1970s by by Erd\H{os}, Rubin, and Taylor \cite{erdos1979choosability}, and independently, by Vizing \cite{vizing}.

\begin{definition}
A \emph{list assignment $L$} for a graph $G$ is a function that assigns to each $v \in V(G)$ a list $L(v)$ of colours.  $L$ is a \emph{$k$-list assignment} if $|L(v)| \geq k$ for every $v \in V(G)$. An \emph{$L$-colouring} of $G$ is a function $\phi$ such that $\phi(v) \in L(v)$ for each $v\in V(G)$, and $\phi(u) \neq \phi(v)$ for each $uv \in E(G)$. We say $G$ is \emph{$L$-colourable} if there exists an $L$-colouring of $G$, and that $G$ is \emph{$k$-list-colourable} (or \emph{$k$-choosable}) if $G$ is $L$-colourable for every $k$-list assignment $L$ for $G$.
\end{definition}

As compared to ordinary colouring, we think of a list assignment as \emph{localizing} the possible images of the colouring function to each vertex.

Thomassen famously proved that every planar graph is 5-list colourable, settling a conjecture posed by Vizing \cite{vizing} and Erd\H{o}s, Rubin, and Taylor \cite{erdos1979choosability}.
\begin{thm}[Thomassen, \cite{thomassen5LC}]\label{5choosable}
Every planar graph is 5-list-colourable.
\end{thm}

This theorem is best possible for planar graphs: in 1993, Voigt gave a construction of a planar graph that is not 4-list-colourable \cite{voigt1993list}. It is an easy consequence of Euler's formula for graphs embedded in surfaces that planar graphs of girth at least four (i.e. without triangles) are 4-list colourable; again, Voigt showed this is best possible in the sense that there exists a planar graph of girth at least four that is not 3-list-colourable \cite{voigt1995not}. When we rule out both triangles and 4-cycles, however, lists of size three suffice.

\begin{thm}[Thomassen, \cite{thomassen3LC}]\label{3choosable}
Every planar graph of girth at least five is 3-list-colourable.
\end{thm}

It is natural to wonder whether these results carry over to graphs embedded in surfaces other than the sphere. To partially answer this question, we require the following definitions.

\begin{definition}
A \emph{non-contractible cycle} in a surface is a cycle that cannot be continuously deformed to a single point. An embedded graph is \emph{$\rho$-locally planar} if every cycle (in the graph) that is non-contractible (in the surface) has length at least $\rho$.
\end{definition}
This is closely related to the concept of \emph{edge-width}. Recall that the edge-width of an embedded graph is the length of the shortest non-contractible cycle; thus if a graph is $\rho$-locally planar, it has edge-width at least $\rho$.

 In 2006, DeVos, Kawarabayashi, and Mohar \cite{devos2006locally} showed that for every surface $\Sigma$, there exists a constant $\rho= 2^{O(g)}$, where $g$ is the Euler genus of $\Sigma$, such that every $\rho$-locally planar graph that embeds in $\Sigma$ is 5-list-colourable. A similar result for 5-colourability (rather than list-colourability) was proved by Thomassen in 1993 \cite{thomassen1993five}. Per the work of Postle and Thomas \cite{postle2018hyperbolic}, analogous results for 5-list-colouring, 4-list-colouring of graphs of girth at least four, and 3-list-colouring of graphs of girth at least five (with $\rho = \Omega(\log(g))$) are implied by the \emph{hyperbolicity} of certain associated families of graphs. \emph{Hyperbolicity} is defined below; $(G, \Sigma)$ is a graph $G$ embedded in a surface $\Sigma$.

\begin{definition}\label{def:hyp}
Let $\mathcal{F}$ be a family of embedded graphs. We say that $\mathcal{F}$ is \emph{hyperbolic} if there exists a constant $c > 0$ such that if $(G, \Sigma) \in \mathcal{F}$ is an embedded graph, then for every closed curve $\eta : S^1 \rightarrow \Sigma$ that bounds an open disk $\Delta$ and intersects $G$ only in vertices, if $\Delta$ includes a vertex of $G$, then the number of vertices of $G$ in $\Delta$ is at most $c(|\{x \in S^1 : \eta(x) \in V (G)\}| - 1)$. We say that $c$ is a \emph{Cheeger constant} for $\mathcal{F}$.
\end{definition}

In \cite{postle2018hyperbolic}, Postle and Thomas give a theorem known as the \emph{hyperbolic structure theorem}, which characterises the structure of graphs in hyperbolic families. We state the theorem below informally, both to help give the reader intuition regarding hyperbolicity and to better explain the implications of hyperbolicity for locally planar graphs. For more information (and a more formal description of what is meant below), we encourage the reader to consult \cite{postle2018hyperbolic}.

\begin{thm}[Theorem 6.29, \cite{postle2018hyperbolic} (informally stated)]\label{hypstrthm}
Let $\mathcal{F}$ be a hyperbolic family of embedded graphs, and let $(G, \Sigma) \in \mathcal{F}$. Let $g$ be the Euler genus of $\Sigma$. The graph $G$ decomposes into a graph with $O(g)$ vertices together with a set of $O(g)$ cylinders of edge-width $O(1)$.
\end{thm}

This theorem together with the hyperbolicity of certain families of graphs implies list-colouring results for locally planar graphs. To explain this further, we again require a few definitions.

\begin{definition}
Let $G$ be a graph, $k$ a positive integer, and $L$ a $k$-list assignment for $G$. We say $G$ is \emph{$L$-critical} if every proper subgraph of $G$ admits an $L$-colouring, but $G$ itself does not. If there exists a $k$-list assignment $L'$ such that $G$ is $L'$-critical, we say $G$ is \emph{critical for $k$-list-colouring}.
\end{definition}

In 2013, Dvo\v{r}\'ak and Kawarabayashi \cite{dvovrak2013list} showed the family of embedded graphs of girth at least five that are critical for 3-list-colouring is hyperbolic. Postle and Thomas showed the same for the family of embedded graphs of girth at least four that are critical for 4-list-colouring \cite{postle2018hyperbolic}; and in 2016 \cite{postle2016five}, for the family of embedded graphs that are critical for 5-list-colouring.

Theorem \ref{hypstrthm} together with these hyperbolicity results is enough to prove that, given a surface $\Sigma$ with genus $g$, for each $k\in \{3,4,5\}$ there exists an integer $\rho$ with $\rho = O(g)$ such that $\rho$-locally planar graphs embeddable in $\Sigma$ are $k$-list-colourable. In \cite{postle2018hyperbolic}, Postle and Thomas showed that with more work, hyperbolicity in fact implies analogous results with $\rho = \Omega(\log(g))$ instead of $O(g)$. As discussed in \cite{postle2018hyperbolic}, this bound is best possible.

We are interested in generalizing these results to the framework of \emph{correspondence colouring.} Correspondence colouring is a natural generalization of list colouring introduced by Dvo{\v{r}}{\'a}k and Postle in 2018 \cite{dvovrak2018correspondence}. Since then, it has been extensively studied: see for example \cite{abe2021differences, bernshteyn2016asymptotic,bernshteyn2018sharp,bernshteyn2019differences,bernshteyn2017dp,liu2019dp,zhang2021edge}.
It is defined as follows.
\begin{definition}
 Let $G$ be a graph. A \emph{$k$-correspondence assignment for $G$} is a $k$-list assignment $L$ together with a function $M$ that assigns to every edge $e = uv \in E(G)$ a partial matching $M_e$ between $\{u\}\times L(u)$ and $\{v\}\times L(v)$.
An $(L,M)$-colouring of $G$ is a function $\varphi$ that assigns to each vertex $v \in V(G)$ a colour $\varphi(v) \in L(v)$ such that for every $e = uv \in E(G)$, the vertices $(u, \varphi(u))$ and $(v, \varphi(v))$ are non-adjacent in $M_e$. We say that $G$ is $(L,M)$-colourable if such a colouring exists, and that $G$ is \emph{$k$-correspondence-colourable} if $G$ is $(L,M)$-colourable for every $k$-correspondence assignment $(L,M)$ for $G$.
\end{definition}

Below, we generalize the notion of criticality to correspondence colouring.
\begin{definition}\label{def:critcorcol} Let $G$ be a graph, $k$ a positive integer, and $(L,M)$ a $k$-correspondence assignment for $G$. We say $G$ is \emph{$(L,M)$-critical} if every proper subgraph of $G$ admits an $(L,M)$-colouring, but $G$ itself does not. If there exists a $k$-correspondence assignment $(L',M')$ such that $G$ is $(L',M')$-critical, we say $G$ is \emph{critical for $k$-correspondence colouring}.
\end{definition}

A correspondence assignment can be thought of as a further localization of colouring: just as list colouring localizes the notion of what colours are available at a vertex, a correspondence assignment localizes the \emph{meaning} of these colours. Many list-colouring theorems carry over to the correspondence colouring framework with only semantic modifications. For instance, as pointed out by Dvo{\v{r}}{\'a}k and Postle \cite{dvovrak2018correspondence}, Theorems \cite{thomassen5LC} and \cite{thomassen3LC}  hold for correspondence colouring: planar graphs are 5-correspondence-colourable, and planar graphs of girth at least five are 3-correspondence-colourable.

The main result of this paper is a technical theorem (Theorem \ref{theorem:stronglinear}) that implies the following.
\begin{restatable}{thm}{5cchyperbolic}\label{5cchyperbolic}
The family of embedded graphs that are critical for 5-correspondence colouring is hyperbolic.  
\end{restatable}

Theorem \ref{theorem:stronglinear}\textemdash which implies Theorem \ref{5cchyperbolic}\textemdash uses similar ideas to that of the analogous theorem for list colouring of Postle and Thomas (Theorem 4.6, \cite{postle2016five}); however, a number of new ideas and reductions are needed in order to make the proof go through in the correspondence colouring framework. This is discussed further in Section \ref{sec:challenges}.

Per the work of Postle and Thomas \cite{postle2018hyperbolic}, Theorem \ref{5cchyperbolic} implies the following.
\begin{restatable}{thm}{locallyplanar}\label{locallyplanar}
For every surface $\Sigma$, there exists a constant $\rho > 0$ such that every $\rho$-locally planar graph that embeds in $\Sigma$ is 5-correspondence-colourable.
\end{restatable}

We note that this result is new, and without hyperbolicity, it is unclear how one would prove it. For the theorem above, $\rho = \Omega(\log(g))$ where $g$ is the Euler genus of $\Sigma$. See Section \ref{sec:implications} for further details. We note this bound for $\rho$ is best possible: since (as noted above) this bound is best possible for list colouring, it immediately follows that it is best possible for correspondence colouring.

Along with this implication for locally planar graphs, hyperbolicity implies a host of other interesting results (as shown by Postle and Thomas \cite{postle2018hyperbolic}). We highlight two more of these below. In a follow-up paper, we will demonstrate how Theorem \ref{theorem:stronglinear} (which implies that the family of embedded graphs that are critical for 5-correspondence colouring is hyperbolic) can also be used to show that planar graphs have exponentially many 5-correspondence colourings, proving a conjecture of Langhede and Thomassen \cite{langhede2021exponentially}. Our method can also be used to prove analogous bounds on the number of list-colourings or correspondence colourings of other classes of planar graphs, assuming the existence of a theorem analogous to Theorem \ref{theorem:stronglinear}. However, in these instances other (often better) bounds were already known (see for example \cite{bosek2022graph} and, more recently, \cite{dahlberg2023algebraic} for lower bounds on the number of 3-list colourings and 3-correspondence colourings, respectively, of planar graphs of girth 5). This is not the case for 5-correspondence colouring. 

In \cite{postle2016five}, Postle and Thomas show that the list-colouring analogue to Theorem \ref{theorem:stronglinear} has implications for the \emph{precolouring extension problem} for planar graphs. The problem can be stated as follows: given a planar graph $G$ with list (or correspondence) assignment $L$ (or $(L,M)$) and subgraph $C$ of $G$, when does an arbitrary $L$ (or $(L,M)$) colouring of $C$ extend to $G$? 

One way to approach the problem is to try to quantify the amount of computation required to determine whether or not the colouring will extend. In particular: can we bound the size of a subgraph $H$ with $H \subseteq G$ such that every colouring of $C$ that extends to $H$ also extends to $G$? Note that a subgraph $H$ with this property always exists: $G$ is such a subgraph. To limit the computation required to answer the decidability question presented above, it is useful to study the minimal subgraphs $H$ with this property. These subgraphs serve as small certificates for the decidability problem.

Postle and Thomas show the following, settling a conjecture of Dvo\v r\'ak et al.~\cite{dvovrak20175}.
\begin{restatable}[Postle and Thomas, \cite{postle2016five}]{thm}{lukethm}\label{lukethm}
Let $G$ be a plane graph with outer cycle $C$, let $L$ be a 5-list assignment for $G$, and let $H$ be a minimal subgraph of $G$ such that every $L$-colouring of $C$ that extends to an $L$-colouring of $H$ also extends to an $L$-colouring of $G$. Then $H$ has at most $19|V(C)|$ vertices.
\end{restatable}

In 1997, Thomassen \cite{thomassen1997color} proved a similar theorem for ordinary colouring, showing $|V(H)| \leq 5^{|V(C)|^3}$. In 2010, Yerger \cite{yerger2010color} improved Thomassen's bound to $O(|V(C)|^3)$. We note that a linear bound in terms of the number of vertices in the precoloured subgraph is asymptotically best possible.

In 2011, Dvo\v r\'ak and Kawarabayashi gave an analogous theorem to Theorem \ref{lukethm} for 3-list-colouring below.
\begin{thm}[Dvo\v r\'ak and Kawarabayashi, \cite{dvorak2011choosability}]\label{dvorakcriticalbound} 
Let $G$ be a plane graph of girth at least five and with outer cycle $C$, let $L$ be a 3-list assignment for $G$, and let $H$ be a minimal subgraph of $G$ such that every $L$-colouring of $C$ that extends to an $L$-colouring of $H$ also extends to an $L$-colouring of $G$. Then $H$ has at most $\frac{37}{3}|V(C)|$ vertices.
\end{thm}

These theorems suggest that, given these graphs and list assignments, there is a small subgraph $H$ that encodes the answer to the precolouring extension problem for cycles: that is, if a cycle $C$ in a plane graph $G$ is precoloured and we wish to determine whether this colouring extends to $G$, there exists a small subgraph $H$ such that it suffices to check whether the colouring extends to $H$. 

We show in Section \ref{sec:implications} that Theorem \ref{theorem:stronglinear} implies the following result.
\begin{restatable}{thm}{theoremlinearcycle}\label{theorem:linearcycle}
Let $G$ be a plane graph with outer cycle $C$, let $(L,M)$ be a 5-correspondence assignment for $G$, and let $H$ be a minimal subgraph of $G$ such that every $(L,M)$-colouring of $C$ that extends to an $(L,M)$-colouring of $H$ also extends to an $(L,M)$-colouring of $G$. Then $H$ has at most $51 |V(C)|$ vertices.
\end{restatable}

The final implications of hyperbolicity that will be discussed in Section \ref{sec:implications} involve algorithms for the decidability of the colouring problem for embedded graphs: Dvo\v r\'ak and Kawarabayashi \cite{dvovrak2013list} also gave linear-time\footnote{The algorithms' running times are linear with respect to the number of vertices in the graph.} algorithms for the decidability of 3-list-colouring of embedded graphs of girth at least five.  Their algorithms can be modified to allow the precolouring of a subgraph $H$, at the cost of increasing the time complexity of the algorithm to $O(|V(G)|^{k(g+s)+1})$ where $k$ is some absolute constant,  $g$ is the genus of the surface in which the graph is embedded, and $s$ is the number of components in $H$. This modification ensures the algorithms find a colouring, should it exist. Theorem \ref{hypstrthm} helps guide the structure of the algorithms: the algorithms roughly attempt to decompose embedded graphs into subgraphs as described in Theorem \ref{hypstrthm}, and find colourings that extend to these subgraphs via dynamic programming. For details, see \cite{dvovrak2013list}. The algorithms rely on Theorem \ref{dvorakcriticalbound}; and per \cite{dvovrak2013list}, these algorithms can be adapted to other settings where a linear bound analogous to that in Theorem \ref{dvorakcriticalbound} holds.

In particular, Theorem \ref{lukethm} thus implies the existence of linear algorithms for deciding the 5-list-colouring of embedded graphs, and Theorem \ref{theorem:linearcycle} implies the following.

\begin{restatable}{thm}{algone}
\label{alg1}
Let $\Sigma$ be a fixed surface. There exists a linear-time algorithm that takes as input an embedded graph $(G, \Sigma)$ and 5-correspondence assignment $(L,M)$ for $G$ with lists of bounded size and determines whether or not $G$ is $(L,M)$-colourable.
\end{restatable}
\begin{restatable}{thm}{algtwo}\label{alg2}
Let $\Sigma$ be a fixed surface. There exists a linear-time algorithm that takes as input an embedded graph $(G, \Sigma)$ and determines whether or not $G$ is 5-correspondence-colourable.
\end{restatable}
Note that in Theorem \ref{alg1} the correspondence assignment $(L,M)$ is fixed, whereas in Theorem \ref{alg2} it is not. We note that these algorithmic results are new; and in fact, prior to this paper, it was not known whether there existed poly-time algorithms (let alone linear algorithms) for the decidability of 5-correspondence colouring embedded graphs.

As mentioned prior, we obtain Theorem \ref{theorem:linearcycle} as a consequence of a more technical theorem (Theorem \ref{theorem:stronglinear}), the proof of which constitutes the bulk of Section \ref{sec:mainthm}. We delay the statement of Theorem \ref{theorem:stronglinear} until Subsection \ref{subsec:mainthmstatement}, when we will have built up the necessary background and terminology.

We further observe in Section \ref{sec:ggeq5} that the embedded graphs $G$ of girth at least five that are critical for 3-correspondence colouring form a hyperbolic family. This follows from observing that the proof for list colouring in \cite{postle20213} also holds for correspondence colouring with only minor modifications. This is discussed further in Section \ref{sec:ggeq5}.

As discussed above, the hyperbolicity of this family of graphs (as well as related theorems) has many interesting implications.  As in the case for 5-correspondence colouring, we highlight the following three.

\begin{restatable}{thm}{algonefive}
\label{alg15}
Let $\Sigma$ be a fixed surface. There exists a linear-time algorithm that takes as input an embedded graph of girth at least five $(G, \Sigma)$ and a 3-correspondence assignment $(L,M)$ for $G$ with lists of bounded size and determines whether or not $G$ is $(L,M)$-colourable.
\end{restatable}
\begin{restatable}{thm}{algtwofive}\label{alg25}
Let $\Sigma$ be a fixed surface. There exists a linear-time algorithm that takes as input an embedded graph of girth at least five $(G, \Sigma)$ and determines whether or not $G$ is 3-correspondence-colourable.
\end{restatable}
\begin{restatable}{thm}{locallyplanarfive}\label{locallyplanar5}
For every surface $\Sigma$, there exists a constant $\rho > 0$ such that every $\rho$-locally planar graph of girth at least five that embeds in $\Sigma$ is 3-correspondence-colourable.
\end{restatable}

Subsection \ref{subsec:outline}, below, gives an outline of the rest of the paper.

%---------------------------------------------------------------------------------------------
%--------------------------------------------OUTLINE OF PAPER
%---------------------------------------------------------------------------------------------
\subsection{Outline of Paper}\label{subsec:outline}

In Section \ref{sec:challenges}, we discuss as a high level some of the challenges involved in the proof of Theorem \ref{theorem:stronglinear}, and in particular, where our proof differs from the analogous theorem for list colouring in \cite{postle2016five}. In Subsection \ref{subsec:criticalgraph}, we establish a few basic results and definitions used in the proof of Theorem \ref{theorem:stronglinear}.  Section \ref{sec:prelims} contains three subsections: Subsection \ref{subsec:criticalgraph} introduces \emph{critical canvases}, our main object of study. Subsection \ref{subsec:deficiency} introduces the notion of \emph{deficiency}, a measurement used throughout the paper. Subsection \ref{subsec:mainthmstatement} establishes yet more useful definitions and results, and concludes with the statement of our main theorem, Theorem \ref{theorem:stronglinear}, which in turn implies Theorem \ref{theorem:linearcycle}, Theorem \ref{5CCconstant}, and Theorem \ref{hyperbolic5cc}: that the family of graphs that are critical for 5-correspondence colouring is hyperbolic. Many proofs in these sections are taken from \cite{postle2016five}, where they were originally written for list colouring. When the list colouring proof directly carries over to the correspondence colouring framework without modification, we omit the proofs in the interest of brevity and refer the reader to 
 \cite{postle2016five}. In these cases, the omitted proofs are purely structural and do not mention colourings. Section \ref{sec:mainthm} contains the proof of Theorem \ref{theorem:stronglinear}. Many of the proofs in this section differ from the analogous results in \cite{postle2016five}. In particular, starting in Subsection \ref{subsec:thirdlayer}, our proof diverges completely from that of Theorem 4.6 in \cite{postle2016five}.  Section \ref{sec:implications} establishes several important consequences of Theorem \ref{theorem:stronglinear}; finally, Section \ref{sec:ggeq5} discusses analogous results for 3-correspondence colouring graphs of girth at least five.

%Section \ref{sec:implications} makes use of the following theorem, due to Thomassen.
%\begin{thm}[Thomassen \cite{thomassen5LC}]\label{tech5CC} Let $G$ be a planar graph with outer face boundary walk $C$. Let $S$ be a path of length at most one contained in $C$. Let $(L,M)$ be a correspondence assignment for $G$ where $|L(v)| \geq 5$ for all $v \in V(G) \setminus V(C)$, and where $|L(v)| \geq 3$ for all $v \in V(C) \setminus V(S)$. Every $(L,M)$-colouring of $S$ extends to an $(L,M)$-colouring of $G$. 
%\end{thm}
%Thomassen originally stated this for list colouring. However, as pointed out by Dvo{\v{r}}{\'a}k and Postle in \cite{dvovrak2018correspondence}, the proof carries over to correspondence colouring.

%---------------------------------------------------------------------------------------------------
%---------------------------------------------------------------------------------------------------
%--------------------------HYPERBOLICITY----------------------------------------------------
%---------------------------------------------------------------------------------------------------
%---------------------------------------------------------------------------------------------------
\section{Challenges and Main Ideas in the Proof of Theorem \ref{theorem:stronglinear}}\label{sec:challenges}

Our proof of Theorem~\ref{theorem:stronglinear}, the main result of this paper, follows the basic framework laid out by Postle and Thomas in \cite{postle2016five} to prove the analogous theorem for list colouring. As discussed in \cite{postle2016five}, the main idea is to bound the number of vertices in a critical graph in terms of the sum of the sizes of large faces: this is the concept Postle and Thomas call ``deficiency''. The proof of Theorem \ref{theorem:stronglinear} also involves counting the number of vertices that share an edge or a face with vertices in the outer cycle of the graph. In keeping track of these quantities, we are able to perform various reductions, showing a minimum counterexample to Theorem \ref{theorem:stronglinear} must have a very specific structure and ultimately that a minimum counterexample cannot exist.

The bulk of the arguments present in the proof of Postle and Thomas' list colouring version of Theorem \ref{theorem:stronglinear} carry over to correspondence colouring with only minor modifications. This is largely due to the fact that many of the arguments are structural, and do not rely on the specific list assignment. However, there are a few key points at which the arguments fail for correspondence colouring. In particular, Claims 5.23 and 5.24 in \cite{postle2016five} argue that the lists of specific vertices in a minimum counterexample are subsets of one another. For a triangle $ux_2z_2u$ in a minimum counterexample with list assignment $S$, Claim 5.23 shows that $S(u) \subseteq S(x_2)$.  Claims 5.27 and 5.28 then use the fact that $S(z_2) \setminus (S(x_2) \cup S(u)) = S(z_2) \setminus S(x_2)$. This (along with an argument showing $S(z_2) \setminus S(x_2)$ is non-empty) implies that it is possible to colour $z_2$ from $S(z_2)$ while avoiding the lists of both $x_2$ and $u$. This argument crucially does not hold for correspondence colouring: an analogous argument to that in Claim 5.23 shows merely that for a correspondence assignment $(S,M)$, we have $|M_{x_2u}| = |S(u)|$, which of course implies nothing about $M_{z_2u}$. As a consequence of this, we are unable to use the reductions found in \cite{postle2016five}, and must instead develop an entirely new set of reductions that can be performed in the correspondence colouring framework. This adds considerable length and intricacy to the proof.

The proof of Theorem \ref{theorem:stronglinear} has two main parts: the first involves purely the structure of a minimum counterexample $G$ to Theorem \ref{theorem:stronglinear}, and the second involves arguing about the specific matchings $\{M_e: e \in E(G)\}$ of the correspondence assignment $(L,M)$ of $G$. Several of our structural results are taken directly from the analogous theorems for list colouring in \cite{postle2016five}: whenever possible, we omit these purely structural proofs in the interest of brevity, referring the reader instead to \cite{postle2016five}. Some of these arguments involve the set of vertices $X_1$ that have at least three neighbours in the outer cycle $C$ of $G$, as well as the set of vertices $X_2$ with at least three neighbours in $V(C) \cup X_1$ and at least one neighbour in $X_1$. We note that the fact that $X_1 \neq \emptyset$ is an easy consequence of Theorem \ref{tech5CC}, a technical theorem due to Thomassen that implies that planar graphs are 5-correspondence colourable. Moreover, as we are able to show that vertices in $X_1$ have exactly three neighbours in $V(C)$; that vertices not on the outer face boundary of $G$ have degree at least five; an that no edge in $G$ has both endpoints in $X_1$, it follows similarly from Theorem \ref{tech5CC} that $X_2$ is non-empty. Informally, we think of the sets $X_1$ and $X_2$ as ``layers'' near the outer cycle $C$. These two layers alone do not provide us with enough freedom to force a contradiction in the second part of the proof, unlike in the proof of the analogous theorem for list colouring given in \cite{postle2016five}. Our analysis thus involves moving one layer further into the graph, and considering the structure surrounding vertices in the set $X_3$ of vertices with at least three neighbours in $V(C) \cup X_1 \cup X_2$ and at least one neighbour in $X_2$. That $X_3$ is non-empty follows from similar reasoning as $X_2$. The proof before this point is very similar to that of Postle and Thomas in \cite{postle2016five}. It is from this point on \textemdash the introduction of this third ``layer'', $X_3$, in Subsection \ref{subsec:thirdlayer} \textemdash that the proof diverges substantially. From this point on, the lemmas and other results in the proof of Theorem \ref{theorem:stronglinear} have no analogues in \cite{postle2016five}.

In the second part of the proof, we argue about the matchings in the correspondence assignment. In particular, Claim \ref{keyclaim} establishes very precisely the matchings between vertices $x_1 \in X_1$, $x_2 \in X_2$, and $x_3 \in X_3$ as well as their other neighbours in the graph. We use this claim to finish the proof, showing that for one edge $e \in E(G)$, we have that $M_e$ is not a matching. This contradicts the definition of correspondence assignment, and thus dispels the existence of a minimum counterexample to Theorem \ref{theorem:stronglinear}.

%======================================================================
%======================================================================
\section{Preliminaries}\label{sec:prelims}
In this section, we establish a few basic definitions and results that will be used throughout the rest of the paper. As mentioned above, several results in this section are already proved in \cite{postle2016five} for list colouring instead of correspondence colouring. In all cases, the results are easily adapted for correspondence colouring. In the interest of brevity, we have omitted the proofs that are identical to those given in \cite{postle2016five}. As mentioned in Section \ref{sec:challenges}, whenever this is the case, the proofs in question are purely structural and do not mention colourings at all.

\subsection{Critical Subgraphs}\label{subsec:criticalgraph}
In this subsection, we establish a few basic properties of \emph{critical canvases}, the main object under study. We will need the following definitions.

\begin{definition}
Let $G$ be a graph. For a set $X \subseteq V(G)$, we denote by $N(X)$ the set $\left(\bigcup_{v \in X} N(v) \right) \setminus X$. 
\end{definition}

\begin{definition}[$S$-critical]
Let $G$ be a graph, $S \subseteq G$ a subgraph of $G$, and $(L,M)$ a correspondence assignment for $G$. For an $(L,M)$-colouring $\phi$ of $S$, we say that \emph{$\phi$ extends to an $(L,M)$-colouring} of $G$ if there exists an $(L,M)$-colouring $\psi$ of $G$ such that $\phi(v)=\psi(v)$ for all $v\in V(S)$.  The graph $G$ is \emph{$S$-critical with respect to $(L,M)$} if $G \ne S$ and for every proper subgraph $G' \subset G$ such that $S \subseteq G'$, there exists an $(L,M)$-colouring of $S$ that extends to an $(L,M)$-colouring of $G'$, but does not extend to an $(L,M)$-colouring of $G$. If the list assignment is clear from the context, we shorten this and say that $G$ is $S$-critical.
\end{definition}

\begin{definition}
\label{def:canvas}
We say the triple $(G,C,(L,M))$ is a \emph{canvas} if $G$ is a $2$-connected plane graph, $C$ is its outer cycle, and $(L,M)$ is a correspondence assignment for the vertices of $G$ such that $|L(v)|\ge 5$ for all $v\in V(G)\setminus V(C)$ and there exists an $(L,M)$-colouring of $C$. We say a canvas $(G,C,(L,M))$ is \emph{critical} if $G$ is $C$-critical with respect to the correspondence assignment $(L,M)$.
\end{definition}
These definitions match those given for list colouring in \cite{postle2016five}, with the appropriate adjustments for correspondence colouring instead of list colouring.

In addition to being used below to establish helpful corollaries regarding subgraphs of critical canvases, the following lemma will be used in Section \ref{sec:implications} to show that the family of embedded graphs that are critical for 5-correspondence colouring form a hyperbolic family.

\begin{lemma}[Lemma 2.3, \cite{postle2016five}] \label{SComponent}
Let $T$ be a subgraph of a graph $G$ that is $T$-critical with respect to the correspondence assignment $(L,M)$. Let $G=(A,B)$ be a separation of $G$ such that $T\subseteq A$ and $B\ne \emptyset$. Then $G[V(B)]$ is $A[V(A)\cap V(B)]$-critical.
\end{lemma}
%\begin{proof}
%Let $G'=G[V(B)]$ and $S=A[V(A)\cap V(B)]$. Since $G$ is $T$-critical, every isolated vertex of $G$ belongs to $T$, and thus every isolated vertex of $G'$ belongs to $S$. Suppose for a contradiction that $G'$ is not $S$-critical. Then, there exists an edge $e \in E(G') \setminus E(S)$ such that every $(L,M)$-colouring of $S$ that extends to $G' \setminus e$ also extends to $G'$. Note that $e \not\in E(T)$. Since $G$ is $T$-critical, there exists a colouring $\Phi$ of $T$ that extends to an $(L,M)$-colouring $\phi$ of $G \setminus e$, but does not extend to an $(L,M)$-colouring of $G$. However, by the choice of $e$, the restriction of $\phi$ to $S$ extends to an $(L,M)$-colouring $\phi'$ of $G'$. Let $\phi''$ be the colouring that matches $\phi'$ on $V(G')$ and $\phi$ on $V(G) \setminus V(G')$. Observe that $\phi''$ is an $(L,M)$-colouring of $G$ extending $\Phi$, which is a contradiction.
%\end{proof}
In this and later sections, we will use of the following theorem, due to Thomassen.
\begin{thm}[Thomassen \cite{thomassen5LC}]\label{tech5CC} Let $G$ be a planar graph with outer face boundary walk $C$. Let $S$ be a path of length at most one contained in $C$. Let $(L,M)$ be a correspondence assignment for $G$ where $|L(v)| \geq 5$ for all $v \in V(G) \setminus V(C)$, and where $|L(v)| \geq 3$ for all $v \in V(C) \setminus V(S)$. Every $(L,M)$-colouring of $S$ extends to an $(L,M)$-colouring of $G$. 
\end{thm}
Thomassen originally stated this for list colouring. However, as pointed out by Dvo{\v{r}}{\'a}k and Postle in \cite{dvovrak2018correspondence}, the proof carries over to correspondence colouring.

Our main theorem characterises planar graphs that are outer cycle-critical (and so planar graphs whose outer face boundary walk is bounded by a cycle). Note that though canvases are 2-connected by definition, the same is not true for critical graphs. The observation below motivates restricting our attentions to 2-connected graphs.
\begin{obs}[Lemma 2.5, \cite{postle2016five}]\label{2connhyp}
Let $G$ be a plane graph with outer cycle $C$, and let $(L,M)$ be a correspondence assignment for $G$ such that $G$ is $C$-critical with respect to $(L,M)$. Then $(G,C,(L,M))$ is a canvas.
\end{obs}
%\begin{proof}
%\textcolor{red}{By the definition of canvas, it suffices to show that $G$ is 2-connected. Suppose not. Then $G$ contains two subgraphs $A$ and $B$ such that $A \cup B = G$, $C \subseteq A$,  $|V(A \cap B)| \leq 1$, and $V(B) \setminus V(A) \neq \emptyset$. By Lemma \ref{SComponent}, we have that $G[V(B)]$ is $A[V(A) \cap V(B)]$-critical. This contradicts Theorem \ref{tech5CC}.} 
%\end{proof}
%\textcolor{blue}{maybe keep}
Before stating the implications of Lemma~\ref{SComponent}, we give the following definition.

\begin{definition}
 Let $T=(G,C,(L,M))$ be a canvas, and let $G'$ be a plane graph obtained from
 $G$ by adding a (possibly empty) set of edges.  If $C'$ is a cycle in $G'$, we let $G\langle C' \rangle$ denote the subgraph of $G\cup C'$ contained in the closed disk bounded by $C'$. We let $T\langle C' \rangle$ denote the canvas $(G \langle C'\rangle, C', (L,M))$.  Similarly, if $G'$ is a subgraph of $G$ and $f$ is a face of $G'$, we denote by $G \langle f \rangle$ the subgraph of $G$ contained in the closed disk given by the boundary walk of $f$, and let $T \langle f \rangle = (G \langle f \rangle,  C_f, (L,M))$, where $C_f$ is the cycle given by the boundary walk of $f$.
\end{definition}

Note that the boundary walk of $f$ is indeed a cycle since $T$ is a canvas (and is thus 2-connected). The following useful corollary follows from Lemma \ref{SComponent}.
\begin{cor}[Proof taken from Corollary 2.7, \cite{postle2016five}] \label{SubCycle}
Let $T=(G,C,(L,M))$ be a critical canvas. If $C'$ is a cycle in $G$ such that $G\langle C' \rangle \ne C'$, then $T\langle C' \rangle$ is a critical canvas.
\end{cor}
\begin{proof}
Let $B=G\langle C' \rangle$ and $A=G\setminus (B\setminus C')$. By Lemma~\ref{SComponent}, it follows that $G\langle C' \rangle$ is $C'$-critical.
\end{proof}

We will require the following definition.
\begin{definition}
Let $T=(G,C,(L,M))$ be a canvas and $G'\subseteq G$ such that $C\subseteq G'$ and $G'$ is $2$-connected. We define the \emph{subcanvas} of $T$ induced by $G'$ to be $(G',C,(L,M))$ and we denote it by $T[G']$.
\end{definition}

Note that in the above definition, the outer cycle of $G'$ is the outer cycle of $G$.

\begin{prop}[Proposition 2.9, \cite{postle2016five}]\label{CriticalSubgraph}
Let $T=(G,C,(L,M))$ be a canvas such that there exists a proper $(L,M)$-colouring of $C$ that does not extend to $G$. Then $T$ contains a critical subcanvas.
\end{prop}

Below, we establish some of the structure of critical canvases. Note Theorem \ref{CycleChordTripod}, below, uses Theorem \ref{tech5CC} in lieu of the nearly identical list-colouring theorem cited in \cite{postle2016five}.

\begin{thm}[Theorem 2.10, \cite{postle2016five}]\label{CycleChordTripod} (Chord or Tripod Theorem)
If $T=(G,C,(L,M))$ is a critical canvas, then either 
\begin{enumerate}
\item $C$ has a chord in $G$, or
\item there exists a vertex of $G$ with at least three neighbours on $C$, and at most one of the internal faces of $G[\{v\}\cup V(C)]$ includes a vertex or edge of $G$.
\end{enumerate}
\end{thm}

The following notation will be useful in proving results about correspondence assignments throughout the paper.
\begin{definition}
Given a graph $G$ with correspondence assignment $(L,M)$, if $uv \in E(G)$ and $(u,d)(v,c) \in M_{uv}$ we write $d = u[v,c]$ and $c = v[u, d]$. We say $d \in L(u)$ \emph{corresponds to $c \in L(v)$} and symmetrically $c \in L(v)$ \emph{corresponds to $d \in L(u)$}. Given $c \in L(v)$, if there does not exist a colour $d \in L(u)$ with $(u,d)(v,c) \in M_{uv}$, we write $u[v,c] = \emptyset$.
\end{definition}

The following easy facts are very useful and will be used throughout the proof of Theorem \ref{theorem:stronglinear}. 
\begin{prop}[Proposition 2.11, \cite{postle2016five}]\label{Facts}
If $T=(G,C,(L,M))$ is a critical canvas, then
\begin{enumerate}
\item for every cycle $C'$ of $G$ of length at most four, $V(G\langle C' \rangle) = V(C')$, and
\item every vertex in $V(G)\setminus V(C)$ has degree at least five.
\end{enumerate}
\end{prop}
%\begin{proof}
%\textcolor{red}{
%We begin with the first statement. Suppose not, and let $C'$ be a cycle in $G$ of length at most four such that $V(G \langle C' \rangle) \neq V(C)$. Let $C' = v_1v_2v_3v_1$ if $C'$ is a triangle, and $C' = v_1v_2v_3v_4v_1$ if $C'$ is a 4-cycle. Since $G$ is $C$-critical, there exists an $(L,M)$-colouring $\phi$ of $C$ that extends to every proper subgraph of $G$ but not to $G$ itself. Then $\phi$ extends to an $(L,M)$-colouring $\phi'$ of $G \setminus V(\Int(C'))$. Let $(L',M')$ be a list assignment for $G' = G[V(\Int[C'])\setminus \{v_3,v_4\}]$ where $L'(v_i) = \phi(v_i)$ for $i \in \{1,2\}$, where $L'(v) = L(v) \setminus \{v[u, \phi(u)]:  u \in \{v_3, v_4\} \cap N(v)\}$ for all $v \in V(G') \setminus V(C')$. Then by Theorem \ref{tech5CC}, $G'$ admits an $(L',M)$-colouring $\phi''$. But $\phi'' \cup \phi'$ is an extension of $\phi$ to $G$, a contradiction.}

%\textcolor{red}{We now prove that every vertex in $V(G) \setminus V(C)$ has degree at least 5. To see this, suppose not: let $v \in V(G) \setminus V(C)$ have degree at most 4. Since $G$ is $C$-critical, there exists an $(L,M)$-colouring $\phi$ of $C$ extends to every proper subgraph of $G$ but not to $G$ itself. Thus $\phi$ extends to an $(L,M)$-colouring $\phi'$ of $G-v$. Since $|L(v)| \geq 5$ and $\deg(v) \leq 4$, it follows that $\phi'$ extends to $v$, a contradiction.}

%\end{proof}

We note that Proposition \ref{Facts} (1) follows from Theorem \ref{tech5CC}. As they are quite straightforward, we omit the short proofs of (1) and (2); they can be found in \cite{evethesis}.

%-------------------------------------------------------------------------------------------
%---------------------------------- DEFICIENCY -------------------------------------
%-------------------------------------------------------------------------------------------
\subsection{Deficiency}\label{subsec:deficiency}
This subsection introduces \emph{deficiency}, a measure defined by Postle and Thomas in \cite{postle2016five}. Our main theorem \textemdash Theorem \ref{theorem:stronglinear} \textemdash concerns the deficiency of critical canvases.

\begin{definition}
Let $G$ be a plane graph, and let $C$ be the subgraph of $G$ whose edge- and vertex-set are precisely those of the outer face boundary walk of $G$. We call a vertex $v \in V(G)$ \emph{internal} if $v \in V(G)\setminus V(C)$. We denote by $v(G)$ the number of internal vertices of $G$. If $T = (G, C, (L,M))$ is a canvas, we define $v(T) = v(G)$. We denote by $\mathcal{F}(G)$ the set of finite faces of $G$; given a face $f$ of $G$, we denote by $|f|$ the length of the boundary walk of $f$. 
\end{definition}
\begin{definition}
Let $G$ be a plane graph, and $H \subseteq G$. The \emph{deficiency of $G$ with respect to $H$} is defined as $\defc(G|H) : = |E(G)\setminus E(H)|-3|V(G)\setminus V(H)|$. When $H$ is clear from context, we sometimes omit it and speak only of the \emph{deficiency of $G$}, denoted $\defc(G)$.  Given a canvas $T = (G,C,(L,M))$, we define $\defc(T) : = \defc(G|C) = |E(G) \setminus E(C)|-3v(G)$. 
\end{definition}

The following lemma will be used in Section \ref{sec:mainthm}.
\begin{lemma}[Lemma 3.4, \cite{postle2016five}]\label{subgraphdefc}
Let $G$ be a 2-connected plane graph with outer cycle $C$, and let $G'$ be a 2-connected subgraph of $G$ containing $C$. Then 
$$\defc(G) = \defc(G') + \sum_{f\in \mathcal{F}(G')} \defc(G \langle f \rangle).$$
\end{lemma}

The following inequality will be helpful in dealing with critical canvases with at most seven internal vertices.
\begin{lemma}[Lemma 3.6, \cite{postle2016five}]\label{equonlyif}
Let $T = (G,C, (L,M))$, where $G$ is a 2-connected plane graph with outer cycle $C$ and every internal vertex of $G$ has degree at least five. Then $$\defc(T) \geq 2v(G)-|E(G\setminus V(C))|,$$ with equality if and only if every vertex of $G$ has degree exactly five.
\end{lemma}

%-------------------------------------------------------------------------------------------
%---------------------------------- LINEAR BOUND -------------------------------------
%-------------------------------------------------------------------------------------------
\subsection{Main Theorem}\label{subsec:mainthmstatement}
In this subsection, we build towards stating Theorem \ref{theorem:stronglinear}, the proof of which constitutes Section \ref{sec:mainthm}. We begin with a few necessary definitions, inherited from \cite{postle2016five}.

\begin{definition}
Let $G$ be a plane graph. We say vertices $u$ and $v$ in $V(G)$ are \emph{cofacial} if there exists a face $f$ of $G$ that is incident to both $u$ and $v$.
\end{definition}

\begin{definition}
Let $G$ be a 2-connected plane graph with outer cycle $C$. We define the \emph{boundary} of $G$, denoted $B(G)$, as $N(V(C))$. We define the \emph{quasi-boundary} of $T$, denoted by $Q(T)$, as the set of vertices not in $C$ that are cofacial with at least one vertex of $C$ (and so $B(T) \subseteq Q(T)$). We let $b(t): =|B(T)|$ and $q(T): = |Q(T)|$. If $T=(G, C, (L,M))$ is a canvas, then we extend the above notions to $T$ in the obvious way, defining $B(T) : = B(G)$ and $Q(T): = Q(G)$.
\end{definition}

From now until Section \ref{sec:ggeq5}, let $\varepsilon$ and $\alpha$ be fixed positive real numbers. Theorem \ref{theorem:stronglinear} depends on $\varepsilon$ and $\alpha$ and holds as long as these two numbers satisfy three inequalities listed in the theorem statement. In Section \ref{sec:implications}, we will make a specific choice of $\varepsilon$ and $\alpha$ in order to optimize the constant in Theorem~\ref{theorem:linearcycle}. Before proceeding, we need one final definition.

\begin{definition} Let $G$ be a 2-connected plane graph with outer cycle $C$. We define $s(G): =\varepsilon \cdot v(G) + \alpha(b(G)+q(G))$ and $d(G): =\defc(G|C)-s(G)$. If $T = (G,C, (L,M))$ is a canvas, we extend these notions to $T$ in the obvious way, defining $s(T):= s(G)$ and $d(T):= d(G)$.
\end{definition}

Below, we establish useful properties of the quantities introduced in the above definitions.

\begin{prop}[Proposition 4.3, \cite{postle2016five}]\label{surplussum}
Let $G$ be a 2-connected plane graph with outer cycle $C$, and let $G'$ be a 2-connected subgraph of $G$ containing $C$ as a subgraph.
\begin{multicols}{2}
\begin{itemize}
\item $v(G)=v(G')+\sum_{f\in\F(G')} v(G\langle f \rangle)$,
\item $b(G)\le b(G')+\sum_{f\in\F(G')} b(G\langle f \rangle)$,
\item $q(G)\le q(G')+\sum_{f\in\F(G')} q(G\langle f \rangle)$,
\end{itemize}
\columnbreak
\begin{itemize}
\item $s(G)\le s(G')+\sum_{f\in\F(G')} s(G\langle f \rangle)$,
\item $d(G)\ge d(G')+\sum_{f\in\F(G')} d(G\langle f \rangle)$.
\item[]  
\end{itemize}
\end{multicols}

\end{prop}

Having defined all necessary quantities, we are now equipped to state our main theorem.

\begin{thm}\label{theorem:stronglinear}
Let $\varepsilon, \alpha, \gamma > 0$ satisfy the following:
\begin{enumerate}
    \item  [(I1)] $2\varepsilon \leq \alpha$  %$3\varepsilon \leq \alpha$,
    \item  [(I2)] $14\alpha + 7\varepsilon \leq \gamma$, and
    \item  [(I3)] $\gamma + 6\alpha + 3\varepsilon \leq 1$ %$\gamma + 2\alpha + 4\varepsilon \leq 1$.
\end{enumerate}
If $T=(G,C,(L,M))$ is a critical canvas and $v(G)\ge 2$, then $d(T)\geq 3-\gamma$.
\end{thm}

%-------------------------------------------------------------------------------------------
%---------------------------------- MAIN PROOF -------------------------------------
%-------------------------------------------------------------------------------------------
\section{Proof of Theorem~\ref{theorem:stronglinear}}\label{sec:mainthm}
This section contains a proof of Theorem \ref{theorem:stronglinear}. Throughout this section, let $T=(G,C,(L,M))$ be a counterexample to Theorem \ref{theorem:stronglinear} such that $|E(G)|$ is minimum;  subject to that, such that $\sum_{v \in V(G)}|L(v)|$ is minimum; and subject to that, such that $\sum_{e \in E(G)} |M_e|$ is maximum. Recall that by Lemma~\ref{Facts}, there is no cycle $C'$ in $G$ of length at most four with $G\langle C'\rangle \neq C'$; and moreover that $\deg(v)\ge 5$ for all internal vertices $v$ of $G$.

The following claim establishes that $G$ contains at least eight internal vertices. A similar claim (showing $v(G) \geq 5$ instead of $v(G) \geq 8$) can be found in \cite{postle2016five} as Claim 5.1. 
\begin{claim}\label{v(T)}
$v(T) \geq 8$.
\end{claim}
\begin{proof}
Suppose not. Note that $s(G) \leq v(G)(\varepsilon + 2\alpha)$, and hence $d(G) \geq \defc(G)-7(2\alpha + \varepsilon)$. Since $14\alpha + 7\varepsilon \leq \gamma$ by (I2) and $T$ is a counterexample to Theorem \ref{theorem:stronglinear}, it follows that $\defc(G) < 3$. Since deficiency is integral, $\defc(G) \leq 2$.

Let $m =|E(G \setminus V(C))|$. By Lemma \ref{equonlyif}, $\defc(G) \geq 2v(G)-m$, and if equality holds every vertex of $G$ has degree exactly 5. With this in mind, we consider the inequality $2 \geq \defc(G) \geq 2v(G)-m$ for each value of $v(G) \in \{2,3, \dots, 7\}$. The cases where $v(G) \leq 5$ are identical to those in \cite{postle2016five}, Claim 5.1; we omit them in the interest of brevity, and assume that $v(G) \in \{6,7\}$.

%\textcolor{red}{Since $2\geq \defc(G) \geq 2v(G)-m$, if $v(G) =2 $ this implies $m \geq 2$. This is a contradiction, since $G$ is a simple graph. Similarly, since $2 \geq \defc(G) \geq 2v(G)-m$, if $v(T)=3$ this implies $m \geq 4$. This too is a contradiction.}  

%\textcolor{red}{Thus we may assume that $v(G) \geq 4$. If $v(G) = 4$, then $2 \geq 2 \cdot 4-m$ and so $m \geq 6$. This implies $G\setminus V(C)$ is a complete graph; and singe $G$ is planar, the outer face boundary walk of $G \setminus V(C)$ is a triangle. This contradicts Proposition \ref{Facts} (1).}

%\textcolor{red}{If $v(T) = 5$, then $2 \geq 2 \cdot 5-m$ and so $m \geq 8$. If $m \geq 9$, then $G \setminus V(C)$ is a triangulation, contradicting Proposition \ref{Facts} (1). Thus we may assume that $m = 8$. It follows that the outer face boundary walk of $G \setminus V(C)$ is not a 5-cycle, as otherwise $m \leq 7$. If it is a cycle of length at most 4, this contradicts Proposition \ref{Facts} (1). Thus the outer face boundary walk of $G \setminus V(C)$ is not a cycle; but then $m \leq 6$, a contradiction.}

%Thus we may assume that $v(G) \in \{6,7\}$.
If $v(T) = 6$, then $2 \geq 2 \cdot 6-m$ and so $m \geq 10$. If $m \geq 11$, the outer face boundary walk of $G \setminus V(C)$ is a 4-cycle, contradicting Proposition \ref{Facts} (1). Thus we may assume $m = 10$, and moreover that the outer face boundary walk of $G \setminus V(C)$ is not a cycle of length at most 4. If it is a cycle of length 6, then $m \leq 9$, a contradiction. If the outer face boundary walk of $G \setminus V(C)$ is not a cycle, then $m \leq 7$, again a contradiction. Thus we may assume the outer face boundary walk of $G \setminus V(C)$ is a 5-cycle; and by Lemma \ref{equonlyif}, every vertex of $G$ has degree exactly 5. Thus $G \setminus V(C)$ is a 5-wheel. We claim that every $(L,M)$-colouring of $C$ extends to $G$. To see this, fix an $(L,M)$-colouring $\phi$ of $C$. By adding a (possibly empty) set of edges to matchings in $M$, we may assume that $|M_{uv}| = \min\{|L(u)|, L(v)|\}$ as this only makes the task of extending a colouring harder. Let the outer cycle of the 5-wheel $G\setminus V(C)$  be $v_1v_2v_3v_4v_5v_1$, and the central vertex $v_6$. Let $S(v_1) : = L(v_1) \setminus \{d: (v_1,d)(u, \phi(u)) \in M_{uv_1} \textnormal{ and } u \in N(v_1) \cap V(C)\}$. By our choice of counterexample $T$ and since $v_1$ has degree 5 in $G$, it follows that $|S(v_1)| = 3$. Thus there exists a choice of colour $c \in L(v_6)$ such that $(c, v_6)(d,v_1) \not \in M_{v_1v_6}$ for all $d \in S(v_1)$. But then $\phi$ extends to $G$ by first colouring $v_6$ with $c$, and then colouring $v_2,v_3,v_4,v_5$, and $v_1$ in that order. This contradicts the fact that $T$ is critical. 

We may therefore assume that $v(T) = 7$. Then $2 \geq 2\cdot 7-m$, and so $m \geq 12$. By Proposition \ref{Facts} (1), the outer face boundary walk of of $G\setminus V(C)$ is not a cycle of length at most 4. Suppose first that it is a 5-cycle.  Then $m \leq 13$. But since $m \geq 12$ and $G$ is planar, at least one internal vertex of $G\setminus V(C)$ does not have degree 5, contradicting Proposition \ref{Facts} (2).  Next, suppose the outer face boundary walk of $G\setminus V(C)$ is a 6-cycle. Them $m \leq 12$, and so $m = 12$. By Lemma \ref{equonlyif}, every vertex in $G \setminus C$ has degree exactly 5. But then $m \leq 11$, a contradiction.  Suppose now the outer face boundary walk of $G\setminus V(C)$ is a 7-cycle. Then $m \leq 11$,  a contradiction. Finally, suppose the outer face boundary walk of $G\setminus V(C)$ is not a cycle. Then $m \leq 10$, again a contradiction.
\end{proof}

%\red{DO we even need this at all? Can we get rid of these? Can we argue from $x_2$, $x_3, $ $x_1$, etc.? We'd probably get better bounds if we can just get rid of this. The colouring arguments don't really need this, so... maybe we're just fine}

\subsection{Proper Critical Subgraphs}

Many of our proofs will involve passing to a smaller canvas whose underlying graph and outer cycle are strictly contained in $G$. It is useful for inductive purposes to be able to bound the deficiency and $d(\cdot)$ of one such canvas in terms of another; the following lemma allows us to do this. As the proof is purely structural, we omit it and refer the reader to \cite{postle2016five}.

\begin{claim}[Claim 5.2, \cite{postle2016five}]\label{ProperCrit}
Suppose $T_0=(G_0,C_0,(L_0,M_0))$ is a critical canvas with $|E(G_0)|\leq |E(G)|$ and $v(G_0)\geq 2$, and let $G'$ be a proper subgraph of $G_0$ such that for some correspondence assignment $(L', M')$, the tuple $(G', C_0, (L',M'))$ is a critical canvas. Then

\begin{enumerate}
\item[(1)] $d(T_0)\geq 4-\gamma$, and
\item[(2)] $d(T_0)\geq 4-2(2\alpha+\varepsilon)$ if $|E(G_0 \setminus E(G')| \geq 2$ and $|E(G')\setminus E(C_0)|\geq 2$, and
\item[(3)] $d(T_0)\geq 5-2\alpha-\varepsilon-\gamma$ if $|E(G_0)\setminus E(G')| \geq 2$ and $|E(G')\setminus E(C_0)|\geq 2$ and $v(G_0)\geq 3$.
\end{enumerate}
\end{claim}
\begin{claim}[Proof taken from Claim 5.3, \cite{postle2016five}]\label{NoProperCrit}
There does not exist a proper $C$-critical subgraph $G'$ of $G$.
\end{claim}
\begin{proof}
This follows from Claim~\ref{ProperCrit} applied to $T_0=T$.
\end{proof}

The following claim will simplify the colouring arguments in Subsection \ref{subsection:tripods}.
\begin{claim}\label{matchmin} If $uv \in E(G) \setminus E(C)$, then $|M_{uv}| = \min\{|L(v)|, |L(u)|\}$.
\end{claim}
\begin{proof}
Suppose not. Then there exist colours $c_1 \in L(u)$ and $c_2 \in L(v)$ such that both $c_1$ and $c_2$ are unmatched in $M_{uv}$. Let $M'$ be obtained from $M$ by setting $M'_e = M_e$ for all $e \neq uv$, and setting $M'_{uv} = M_{uv} \cup \{(u,c_1)(v,c_2)\}$. Let $T' = (G,C,(L,M'))$. Note that $|V(T)| = |V(T')|$, and that the sum of the list sizes of the vertices in $T'$ is the same as that in $T$. Since $(G,C,(L,M))$ was chosen to maximize $\sum_{e \in E(G)} |M_e|$, it follows that $T'$ is not a counterexample to Theorem \ref{theorem:stronglinear}. Otherwise, since $\defc(T') = \defc(T)$, we have that $T'$ contradicts our choice of $T$. Thus  $T'$ is not a critical canvas. Since $T$ is $C$-critical, there exists an $(L,M')$-colouring of $C$ that does not extend to $G$. By Proposition \ref{CriticalSubgraph}, $T'$ contains a critical subcanvas $(G',C,(L,M'))$; and since $T'$ is not critical, $G'$ is a proper subgraph of $G$. But this contradicts Claim \ref{NoProperCrit}.
\end{proof}

\begin{claim}[Proof taken from Claim 5.4, \cite{postle2016five}]\label{Chord}
There does not exist a chord of $C$.
\end{claim}
\begin{proof}
Suppose there exists a chord $e$ of $C$. Let $G'=C\cup e$. As $v(T)\neq 0$, it follows that $G'$ is a proper subgraph of $G$. Yet $G'$ is $C$-critical, contradicting Claim~\ref{NoProperCrit}.  
\end{proof}

\subsection{Dividing Vertices}\label{subsec:dividingvertices}

In this subsection, we establish a few claims regarding \emph{dividing vertices}, defined below. Namely, we show that if $H$ is a critical canvas with at most $|E(G)|$ edges, then if $H$ contains a \emph{true} dividing vertex (Claim \ref{DividingTrue}) or a \emph{strong} dividing vertex (Claim \ref{DividingStrong}) then $d(H)$ is relatively high. These claims will be useful in the following section in showing that certain canvases obtained from $T$ (called \emph{relaxations}) do not contain true or strong dividing vertices. This in turn is useful in arguments showing that when passing to these relaxations, the size of the boundaries and quasiboundaries of the resulting canvases are at least $b(G)$ and $q(G)$. Note all of the nomenclature is inherited from \cite{postle2016five}.

\begin{definition}  Let $G_0$ be a 2-connected plane graph with outer cycle $C_0$. Let $v$ be a vertex in $V(G)\setminus V(C)$, and suppose there exist two distinct faces $f_1,f_2\in \F(G_0)$ such that for $i \in \{1,2\}$ the boundary of $f_i$ includes $v$ and a vertex of $C_0$, say $u_i$. Suppose that $u_1 \neq u_2$, and let $G'$ be the plane graph obtained from $G_0$ by adding the edges $u_1v, u_2v$ if they are not present in $G_0$. Consider the cycles $C_1,C_2$ of $G'$, where $C_1\cap C_2 =u_1vu_2$ and $C_1\cup C_2 = C_0\cup u_1vu_2$. If for both $i\in \{1,2\}$ we have that $|E(T\langle C_i\rangle) \setminus E(C_i)|\ge 2$, then we say that $v$ is a \emph{dividing} vertex. If for both $i\in\{1,2\}$ we have that $|V(T\langle C_i \rangle)\setminus V(C_i)|\ge 1$, we say $v$ is a \emph{strong dividing} vertex. If $v$ is a dividing vertex and the edges $u_1v, u_2v$ are in $G$, then we say that $v$ is a \emph{true dividing} vertex. If $T_0 = (G_0, C_0, (L_0, M_0))$ is a canvas, then by a (true/strong) \emph{dividing} vertex of $T_0$ we mean a (true/strong) dividing vertex of $G_0$.
\end{definition}

\begin{claim}[Claim 5.6, \cite{postle2016five}]\label{DividingTrue}
Suppose $T_0=(G_0,C_0,(L_0, M_0))$ is a critical canvas with $e(G_0)\le e(G)$ and $v(G_0)\ge 2$. If $G_0$ contains a true dividing vertex, then 
\begin{enumerate}
\item $d(T_0)\ge 3-2(2\alpha+\varepsilon)$, and
\item $d(T_0)\ge 4-2\alpha-\varepsilon-\gamma$ if $v(G_0)\ge 3$.
\end{enumerate}
\end{claim}

The following proof is nearly identical to the analogous result in \cite{postle2016five}, with a few minor changes to the colouring arguments to ensure they hold for correspondence colouring.
\begin{claim}[Proof adapted from Claim 5.7, \cite{postle2016five}]\label{DividingStrong}
Suppose $T_0=(G_0,C_0,(L_0, M_0))$ is a critical canvas with $|E(G_0)|\leq |E(G)|$. If $G_0$ contains a strong dividing vertex, then $d(T_0)\ge 4-2\gamma$.
\end{claim}
\begin{proof}
Let $u_1,u_2, C_1, C_2$, and $v$ be as in the definition of strong dividing vertex. Since $v$ is a strong dividing vertex, it follows that $v(T_0) \geq 3$ and $v(T_0 \langle C_1 \rangle) \geq 1$. If $v(T_0 \langle C_1 \rangle) = 1$, then the unique vertex in $V(T_0 \langle C_1 \rangle) \setminus V(C_1)$ is a true dividing vertex and so 
$$
d(T_0) \geq 4-2\alpha-\varepsilon-\gamma \geq 4-2\gamma
$$
by Claim \ref{DividingTrue} and (I2), as desired. So we may assume that $v(T_0 \langle C_1 \rangle) \geq 2$ and similarly that  $v(T_0 \langle C_2 \rangle) \geq 2$.

Let $G_0'$ be the graph obtained from $G_0$ by adding vertices $z_1,z_2$ not in $V(G)$ and edges $u_1z_1$, $z_1v$, $vz_2$, $z_2u_2$. Similarly let $G'$ be the graph obtained from $C_0$ by adding vertices $v, z_1, $, and $z_2$ and edges $u_1z_2,z_1v,vz_2,$ and $z_2v_2$. Let $L_0'(x) = L_0(x)$ for all $x \in V(G_0)$, and let $L_0'(z_1) = L_{0}'(z_2) = R$, where $R$ is a set of five new colours. Let $M_0'$ be defined as $(M'_0)_{uv} = (M_0)_{xy}$ for all $xy \in E(G_0')$ with $\{z_1,z_2\} \cap \{x, y\} = \emptyset$, and $(M'_0)_{xy} = \emptyset$ for all $xy \in E(G_0')$ with $\{z_1, z_2\} \cap \{x,y\} \neq \emptyset$. Let $T_0' = (G_0', C_0, (L_0', M_0')).$ Now
$$
\defc(G') = |E(G')|-|E(C_0)|-3v(G') = 4-3 \cdot 3 = -5.
$$
Since $G_0$ is $C_0$-critical, there exists an $(L_0, M_0)$-colouring $\phi_0$ of $C_0$ that does not extend to $G_0$. By Claim \ref{ProperCrit} the graph $G_0$ does not have a proper $C_0$ critical subgraph, and so by Proposition \ref{CriticalSubgraph}, $\phi_0$ extends to every proper subgraph of $G_0$. For every $c \in L_0(v)$, let $\phi_c(v) = c$, $\phi_c(z_1) = \phi_c(z_2) \in R$, and $\phi_c(x) = \phi_0(x)$ for all $x \in C_0$. Let $C_1'$ and $C_2'$ be the two facial cycles of $G'$ other than $C_0$. Since $\phi_0$ does not extend to an $(L_0, M_0)$-colouring of $G_0$, for every $c \in L_0(v)$ the colouring $\phi_c$ does not extend to an $(L_0', M_0')$-colouring of either $G_0' \langle C'_1 \rangle$ or $G_0' \langle C'_2 \rangle$. Since $|L_0(v)| \geq 5$, there exists $i \in \{1,2\}$ such that there exist at least three colours $c \in L_0(v)$ such that $\phi_c$ does not extend to an $(L_0',M_0')$-colouring of $G_0'$. We may assume without loss of generality that $i=1$. Let $\mathcal{C}$ be the set of all colours $c \in L(v)$ such that $\phi_c$ does not extend to an $(L_0,M_0)$-colouring of $G_0' \langle C_1' \rangle$. Thus $|\mathcal{C}| \geq 3$.

Let $G_1'$ be the graph obtained from $G_0'\langle C_1' \rangle$ by adding the edge $z_1z_2$ inside the outer face of $G_0' \langle C_1' \rangle$. Let $C_1'' = (C_1' \setminus v) + z_1z_2$. Let $L'(z_1) = \{c_1\}$, $L'(z_2) = \{c_2\}$, and $L'(x) = L_0(x)$ for every $x \in V(G_0) \setminus \{v\}$. Let $L'(v) = \mathcal{C} \cup \{c_1,c_2\}$. Let $M'$ be defined as follows: we set $(M'_{xy} = M_0)$ for all $xy \in E(G_0)$ with $\{x,y\} \cap \{z_1, z_2\} = \emptyset$; we set $M'_{xy} = \emptyset$ for all $xy \in E(G_0)$ with $v \not \in \{x,y\}$ and $\{x,y\} \cap \{z_1,z_2\} \neq \emptyset$; and finally we set $M'_{vz_i} = \{(v,c_i)(z_i,c_i)\}$ for $i \in \{1,2\}$.

We claim that $T_1 = (G_1', C_1'', (L',M'))$ is a critical canvas. To see this, let $H$ be a proper subgraph of $G_1'$ that includes $C_1''$ as a subgraph. Let us extend $\phi_0$ by defining $\phi_0(z_1):= c_1$ and $\phi_0(z_2): = c_2$. We will show that $\phi_0$ restricted to $C_1''$ extends to $H$ but not to $G_1'$. If $\phi_0$ extended to $G_1'$, then $\phi_0(v) \not \in \mathcal{C}$ by definition of $\mathcal{C}$ and $\phi_0(v) \not \in \{c_1,c_2\}$ because $v$ is adjacent to both $z_1$ and $z_2$, a contradiction. Thus $\phi_0$ does not extend to $G_1'$. To show that $\phi_0$ extends to $H$ assume first that $H \setminus \{z_1,z_2\}$ is a proper subgraph of $G_0 \langle C_1' \rangle$. Then $(H \setminus \{z_1,z_2\}) \cup G_0 \langle C_2' \rangle$ is a proper subgraph of $G_0$, and hence $\phi_0$ extends to it, as desired. So we may assume that $H \setminus \{z_1,z_2\} = G_0 \langle C_1' \rangle$. Since $H$ is a proper subgraph of $G_1'$ we may assume symmetrically that $vz_1 \in E(H)$. Now $\phi_0$ extends to an $(L',M')$-colouring of $G_0 \setminus \{v\}$. Letting $\phi_0(v) = c_1$ shows that $\phi_0$ extends to $H$, as desired. This proves the claim that $T_1$ is critical. 
As $v(T_1) \geq 2$, we have that $d(T_1) \geq 3-\gamma$ by the minimality of $T$. Similarly, since $v(T_0' \langle C_2'\rangle) \geq 2 $, we have that $d(T_0' \langle C_2'\rangle ) \geq 3-\gamma$. By Lemma \ref{subgraphdefc}, 
$$
\defc(T_0') = \defc(T_0' [G_0']) + \defc(T_0' \langle C_1' \rangle) + \defc(T_0' \langle C_2'\rangle) = -5 +  \defc(T_0' \langle C_1' \rangle) + \defc(T_0' \langle C_2'\rangle).
$$
Yet $\defc(T_0) = \defc(T_0')+2$. Furthermore, $ \defc(T_0' \langle C_1' \rangle) = \defc(T_1)+1$. Hence,
$$
\defc(T_0) =  \defc(T_0' \langle C_1' \rangle) + \defc(T_0' \langle C_2'\rangle) -3 = \defc(T_1) + \defc(T_0' \langle C_2'\rangle) -2.
$$

Moreover, we claim that $s(T_0) \leq s(T_1) + s(T_0' \langle C_2' \rangle)$. This follows from the fact that every vertex of $V(G_0) \setminus V(C_0)$ is either in $V(G_1') \setminus V(C_1'')$ or $V(G_0' \langle C_2' \rangle) \setminus V(C_2')$. Moreover every vertex of $B(T_0)$ is either in $B(T_1)$ or $B(T_0' \langle C_2' \rangle)$ and similarly every vertex of $Q(T_0)$ is either in $Q(T_1)$ or $Q(T_0' \langle C_2' \rangle)$.

Putting this all together, we have that
$$
d(T_0) \geq d(T_1) + d(T_0' \langle C_2' \rangle)-2 \geq 2(3-\gamma)-2 = 4-2\gamma,
$$
as desired.
\end{proof}

\subsection{Tripods}\label{subsection:tripods}
An easy consequence of Theorem \ref{tech5CC} is that if $(G,C,(L,M))$ is a $C$-critical canvas with $v(G) \geq 2$, then there exists a vertex in $V(G) \setminus V(C)$ with at least three neighbours in $V(C)$. Otherwise, every $(L,M)$-colouring of $C$ extends to an $(L,M)$-colouring of $G$, contradicting the fact that $G$ is $C$-critical. Much of our analysis will involve performing certain reductions around specific vertices in $G$ that have precisely three neighbours in $C$, arguing about the deficiency and $d(\cdot)$ of the resulting canvas, and extrapolating from that about $\defc(T)$ and $d(T)$. To that end, it is convenient to define the following. Again, most of our vernacular is inherited from \cite{postle2016five}.

\begin{definition}
Let $G_0$ be a plane graph with outer cycle $C_0$, and let $v \in V(G_0)\setminus V(C_0)$ have at least three neighbours in $C_0$. Let $u_1, u_2, \cdots, u_k$ be all the neighbours of $v$ in $C_0$ listed in a counterclockwise order of appearance on $C_0$. Assume that at most one face of $G_0[V(C_0) \cup \{v\}]$ includes an edge or vertex of $G_0$, and that if such a face exists, then it is incident with $u_1$ and $u_k$. If $k=3$, then we say that $v$ is a \emph{tripod} of $G_0$, and if $k \geq 4$, then we say that $v$ is a \emph{quadpod} of $G_0$. The tripod or quadpod $v$ is \emph{regular} if there exists a face of $G_0[V(C_0) \cup \{v\}]$ that includes an edge or vertex of $G_0$. If such a face exists, then we say that $u_1, u_2, \cdots, u_k$ are listed in \emph{standard (counterclockwise) order}. Note that every tripod of degree at least four is regular.
\end{definition}

If $v$ is a regular tripod or quadpod, we let $C_0 \oplus v$ denote the boundary of the face of $G_0[V(C_0) \cup \{v\}]$ that includes an edge or vertex of $G_0$, and we define $G_0 \oplus v : = G_0 \langle C_0 \oplus v \rangle$. If $X$ is a set of tripods or quadpods of $G_0$ and there exists a face of $G_0[V(C_0) \cup X]$ that includes an edge or vertex of $G_0$, then we let $C_0 \oplus X$ denote the boundary of such a face and define $G_0 \oplus X : = G_0 \langle C_0 \oplus X \rangle$. 

If $T_0 = (G_0, C_0, (L_0, M_0))$ is a canvas, then we extend all the above terminology to $T_0$ in the natural way: thus we can speak of tripods or quadpods of $T_0$, we define $T_0 \oplus X = T_0[G_0 \oplus X]$, etc.

\begin{claim}[Claim 5.9, \cite{postle2016five}] \label{regortrue}
Let $T_0 = (G_0, C_0, (L_0, M_0))$ be a canvas with $v(G_0) \geq 2$, and let $v \in V(G_0)\setminus V(C_0)$ have at least three neighbours in $C_0$. Then $v$ is either a regular tripod of $T_0$, or a true dividing vertex of $T_0$.
\end{claim}
%\begin{proof}
%Let $u_1, u_2, \dots, u_k$ be all the neighbours of $v$ in $C_0$ listed in their order of appearance on $C_0$ and numbered such that the face $f$ of $G_0[V(C_0) \cup \{v\}]$ incident with $u_1$ and $u_k$ includes a vertex of $G_0$. If another face of $G_0[V(C_0) \cup \{v\}]$ includes an edge or vertex of $G_0$ or if $k \geq 4$, then by considering the vertices $u_1$ and $u_k$ we find that $v_0$ is a true dividing vertex of $T_0$. Thus we may assume that $k=0$ and that $f$ is the only face of $G_0[V(C_0) \cup \{v\}]$ that contains a vertex or edge of $G_0$. It follows that $v$ is a tripod, as desired.
%\end{proof}
The idea behind the proof of Claim \ref{regortrue} is that if more than one face of $G_0[V(C_0) \cup \{v\}]$ includes an edge or vertex of $G_0$ or if $v$ has more than three neighbours in $C_0$, then $v$ is a true dividing vertex of $T_0$; and otherwise $v$ is a regular tripod.
\begin{definition}
Let $T_0=(G_0,C_0,(L_0, M_0))$ be a canvas. We say that $T_0$ is a $0$-relaxation of $T_0$. Let $k>0$ be an integer, $T_0'$ be a $(k-1)$-relaxation of $T_0$, and $v$ be a regular tripod for $T_0'$. Then we say that $T_0'\oplus v$ is a \emph{$k$-relaxation} of $T_0$. 
\end{definition}

As pointed out in \cite{postle2016five}, by Proposition \ref{Facts} (2), every tripod of a critical canvas is regular. As a consequence, if $v$ is a tripod of a critical canvas $T_0$, then $T_0 \oplus v$ is well-defined. Moreover, $T_0 \oplus v$ is itself a critical canvas by Corollary \ref{SubCycle}, and $v(T_0 \oplus v) \geq 2 $ by Proposition \ref{Facts} (2). Thus for all $k \geq 1$, a $k$-relaxation $T_0'$ of a critical canvas is well-defined; $T_0'$ is critical; and $v(T_0') \geq 2$. Claims \ref{Relax1} and \ref{dividingrelaxations}, below, establish useful results about canvas relaxations.

\begin{claim}[Claim 5.11, \cite{postle2016five}]\label{Relax1}
If $T_0'$ is a $k$-relaxation of a canvas $T_0$, then $d(T_0)\ge d(T_0') - k(2\alpha+\varepsilon)$.
\end{claim}
%\begin{proof}
%We proceed by induction on $k$.  We may assume that $k \geq 1$ (as otherwise the claim trivially holds) and that the claim holds for all integers strictly smaller than $k$. Let $T_{k-1}$ be a $(k-1)$-relaxation of $T_0$ and $v$ a regular tripod of $T_{k-1}$ such that $T_0'$ is a 1-relaxation of $T_{k-1}$. By induction, $d(T_0)\geq d(T_{k-1})-(k-1)(2\alpha+ \varepsilon)$. Yet $\defc(T_{k-1})= \defc(T_0')$ while $v(T_{k-1}) = v(T_0')+1$, $b(T_{k-1}) \leq b(T_0')+1$, and $q(T_{k-1}) \leq q(T_0')+1$. Thus $d(T_{k-1}) \geq d(T_0')-(2\alpha+\varepsilon)$ and the claim follows.
%\end{proof}
The proof of Claim \ref{Relax1} is by induction on $k$, noting that each time we take a relaxation of a canvas, we lose an internal vertex, and gain at least one boundary vertex (and thus at least one quasi-boundary vertex).

The following claim is analogous to Claim 5.12 in \cite{postle2016five}. However, Claim 5.12 in \cite{postle2016five} makes no mention of the case when $k \geq 3$, or of $k \geq 2$ for strong dividing vertices, as these cases were not necessary for the resulting analysis for list colouring. 
\begin{claim}\label{dividingrelaxations}
Let $k \in \{0,1,2,3,4\}$ and let $T'$ be a $k$-relaxation of $T$. Then $T'$ does not have a true dividing vertex, and if $k \leq 3$, it does not have a strong dividing vertex.
\end{claim}
\begin{proof}
Suppose not. By Corollary \ref{SubCycle}, $T'$ is critical, and $v(T') \geq 3$ by Claim \ref{v(T)}. If $T'$ is a true dividing vertex, then by Claim \ref{Relax1} we have that $d(T) \geq d(T')-4(2\alpha+\varepsilon)$. By Claim \ref{DividingTrue} (2), we have moreover that $d(T') \geq 4-2\alpha-\varepsilon-\gamma$. Thus
\begin{align*}
    d(T) &\geq d(T)-4(2\alpha + \varepsilon)\\
    &\geq 4-2\alpha-\varepsilon-\gamma-8\alpha-4\varepsilon \\
    &= 3-\gamma+(1-10\alpha -7\varepsilon) \\
    &> 3-\gamma
\end{align*}
where the last line follows because $10\alpha+7\varepsilon < 1$ by inequalities (I2) and (I3) (which imply together that $20\alpha + 10\epsilon \leq 1$).
Thus we may assume that $k \leq 3$ and $T'$ has a strong dividing vertex. By Claim \ref{Relax1} we have that $d(T) \geq d(T')-3(2\alpha+\varepsilon)$. By Claim \ref{DividingStrong}, we have that $d(T') \geq 4-2\gamma$. Thus
\begin{align*}
    d(T) &\geq d(T')-3(2\alpha+\varepsilon) \\
    &\geq 4-2\gamma -6\alpha-3\varepsilon \\
    &= 3-\gamma + (1-\gamma-6\alpha -3\varepsilon).
\end{align*}
But by (I3), $\gamma + 6\alpha + 3\varepsilon \leq 1$. Thus $d(T) \geq 3-\gamma$, a contradiction.
\end{proof}

The next claim establishes some of the surrounding structure of tripods of $T$.

\begin{claim}[Claim 5.13, \cite{postle2016five}]\label{QuasiSame}
Let $x_1$ be a tripod of $T$, and let $T'= T \oplus x_1$. Either
\begin{enumerate}
\item $\deg(x_1)=5$, or,
\item $\deg(x_1)=6$, the neighbours of $x_1$ not in $C$ form a path of length two and the ends of that path are in $B(T)$, $b(T)=b(T')$, $q(T)=q(T')$ and $d(T)\ge d(T')-\varepsilon$.
\end{enumerate}
\end{claim}

The following claim is analogous to Claim 5.14 in \cite{postle2016five}; however, the result in \cite{postle2016five} is weaker in that $k \leq 3$. The proof is nearly identical.
\begin{claim}\label{criticalrelaxations}
For $k \in \{0,1,2,3,4,5,6\}$, if $T'$ is a $k$-relaxation of $T$, then there does not exist a proper critical subcanvas of $T'$. 
\end{claim}
\begin{proof}
Suppose not. By Claim \ref{v(T)}, $v(T') \geq 1$. Then by Claim \ref{ProperCrit}, we have that $d(T') \geq 4-\gamma$. By Claim \ref{Relax1}, $d(T) \geq 4-\gamma-6(2\alpha+\varepsilon)$. By (I2), this is at least $3-\gamma$, a contradiction. 
\end{proof}

For the remainder of this section, let $X_1$ be the set of internal vertices of $G$ with at least three neighbours in $C$.   The following claim combines Claims 5.15, 5.16, and 5.17 in \cite{postle2016five}, and establishes more of the structure surrounding vertices in $X_1$. The proof (which we omit) is purely structural, and is very similar to the proof of Claim \ref{X2claims} (which we do include).

\begin{claim}[Claims 5.15-5.17, \cite{postle2016five}]\label{X1claims}
The following all hold.
\begin{enumerate}
    \item $X_1\neq \emptyset$ and every member of $X_1$ is a tripod of $T$.
    \item $T \oplus X_1$ is well-defined and is a critical canvas.
    \item The graph $G \oplus X_1$ does not have a chord of $C \oplus X_1$.
\end{enumerate}
\end{claim}
%\begin{proof}
%We begin by proving (1). By Claim \ref{Chord}, there does not exist a chord of $C$, and hence $X_1 \neq \emptyset$ by Theorem \ref{CycleChordTripod}. By Claims \ref{regortrue} and \ref{dividingrelaxations} every member of $X_1$ is a tripod of $T$.

%Next, we prove (2). By Proposition \ref{Facts} (2) every tripod of $G$ is regular, and hence $T \oplus X_1$ is well-defined. It is critical by Corollary \ref{SubCycle}.

%Finally, we prove (3). Suppose not. Let $v_1v_2$ be a chord of $C \oplus X_1$. As $C$ has no chord by Claim \ref{Chord}, we may assume without loss of generality that $v_1 \not \in V(C)$. Thus $v_1$ is a tripod of $C$. Hence $v_2$ is also a tripod, as otherwise $v_1$ is not a tripod. But then $v_2$ is a true dividing vertex of $T \oplus v_1$ because $v(T) \geq 4$ by Claim \ref{v(T)}, contradicting Claim \ref{dividingrelaxations}.
%\end{proof}

%\begin{claim}\label{X1}[Claim 5.15, \cite{postle2016five}]
%$X_1\neq \emptyset$ and every member of $X_1$ is a tripod of $T$.
%\end{claim}
%\begin{claim} \label{toplusx1}[Claim 5.16, \cite{postle2016five}]
%$T \oplus X_1$ is well-defined and is a critical canvas.
%\end{claim}
%\begin{claim}\label{chordx1}[Claim 5.17, \cite{postle2016five}]
%The graph $G \oplus X_1$ does not have a chord of $C \oplus X_1$.
%\end{claim}

For the remainder of this section, let $X_2$ be the set of vertices $v \in V(G) \setminus (V(C) \cup X_1)$ with at least three neighbours in $C \oplus X_1$. The first part of the following claim corresponds to Claim 5.19 in \cite{postle2016five}; the following two parts are analogous to parts (2) and (3) of Claim \ref{X1claims}, above. Though this additional structure surrounding $X_2$-vertices was not needed for the list colouring version of Theorem \ref{theorem:stronglinear}, it is necessary for our theorem in order to eventually study the properties of a third \say{layer}, $X_3$, in Subsection \ref{subsec:thirdlayer}.

%\begin{claim}[Claim 5.18, \cite{postle2016five}]
%$v(T\oplus X_1) \geq 2$.
%\end{claim}
%\begin{proof}
%Every tripod of $T$ has at least two neighbours in $(G \oplus X_1) \setminus V(C)$ by Proposition \ref{Facts} (2), and no neighbour in $X_1$ by Claim \ref{chordx1}.
%\end{proof}

%\begin{claim}\label{x2}[Claim 5.19, \cite{postle2016five}]
%We have $X_2 \neq \emptyset$. Furthermore, let $x_2 \in X_2$, and let $u_1, u_2, \cdots, u_k$ be the neighbours of $x_2$ in $C \oplus X_1$ listed in standard order. Then $k = 3$, and $u_2 \in V(C)$. In particular, every member of $X_2$ is a tripod \blue{of $C \oplus X_1$}.
%\end{claim}

\begin{claim}\label{X2claims}
 The following all hold.
 \begin{enumerate}
     \item (Claim 5.19, \cite{postle2016five}.) $X_2 \neq \emptyset$. Furthermore, let $x_2 \in X_2$, and let $u_1, u_2, \cdots, u_k$ be the neighbours of $x_2$ in $C \oplus X_1$ listed in standard order. Then $k = 3$, and $u_2 \in V(C)$. In particular, every member of $X_2$ is a tripod of $C \oplus X_1$.
     \item $(T \oplus X_1) \oplus X_2$ is well-defined and is a critical canvas.
     \item The graph $(G \oplus X_1) \oplus X_2$ does not have a chord of $(C \oplus X_1) \oplus X_2$.
 \end{enumerate}
 \end{claim}
 \begin{proof}
We begin by proving (1). By Claim \ref{X1claims} (3), there does not exist a chord of $C \oplus X_1$ and hence from Claim \ref{X1claims} (2) and Theorem \ref{CycleChordTripod} it follows that $X_2 \neq \emptyset$. Let $x_2 \in X_2$, and $u_1, u_2, \dots, u_k$ be as stated. Let $i \in \{2,3, \cdots, k-1\}$. If $u_i \in X_1$, then since $u_i$ has no neighbours in $X_1$ by Claim \ref{X1claims} (3), we have that $u_i$ has three neighbours in $C$ and is adjacent to $x_2$ but has no other neighbours, contrary to Proposition \ref{Facts} (2). Thus $u_i \in V(C)$. 
We may assume that $k \geq 4$, as otherwise (1) holds. Since $x_2 \not \in X_1$, we may assume from the symmetry that $u_1 \in X_1$. By considering the vertices $u_1$ and $u_4$ we find that $x_2$ is a true dividing vertex of either $T \oplus u_1$ (if $u_4 \in V(C)$) or $T \oplus \{u_1, u_4\}$ (if $u_4 \not \in V(C)$), either case contradicting Claim \ref{dividingrelaxations}.
 
We now prove (2). By Claim \ref{X1claims} (2), we have that $T \oplus X_1$ is well-defined. By Claim \ref{X2claims} (1), every member of $X_2$ is a tripod of $C \oplus X_1$, and so by Proposition \ref{Facts} we have that $(T \oplus X_1) \oplus X_2$ is well-defined. It is a critical canvas by Corollary \ref{SubCycle}.
 
Finally, we prove (3). Suppose not. Let $v_1v_2$ be a chord of $(C \oplus X_1) \oplus X_2$. By Claims \ref{Chord} and \ref{X1claims} (2), we may assume without loss of generality that $v_1 \in X_2$. Since every member of $X_2$ is a tripod of $C \oplus X_1$ by (2), it follows that $v_2 \in X_2$. Let $u_1, u_2, u_3$ be the neighbours of $v_1$ in $C \oplus X_1$ listed in standard order, and let $w_1, w_2, w_3$ be the neighbours of $v_2$ in $C \oplus X_1$ listed in standard order. Since $v_1$ and $v_2$ are not in $X_1$, at least one of $u_1$ and $u_3$ (say $u_1$) is in $X_1$; and similarly, at least one of $w_1$ and $w_3$ is in $X_1$.  Let $Y = \{w_1, w_3\} \cap X_1$. First suppose that $\{w_1, w_3\} \cap \{u_1, u_3\} = \emptyset$. Recall that by (1), $u_2$ and $w_2$ are in $V(C)$.Then since $\deg(v_1) \geq 5$ and $\deg(v_2) \geq 5$ by Proposition \ref{Facts} (2), it follows that $v_1$ is a true dividing vertex of $T \oplus u_1 \oplus Y  \oplus v_2$, contradicting Claim \ref{dividingrelaxations}.

Thus we may assume that $\{w_1, w_3\} \cap \{u_1, u_3\} \neq \emptyset$. Without loss of generality, suppose that $u_3 = w_3$. Then again since $\deg(v_1) \geq 5$ and $\deg(v_2) \geq 5$ by Proposition \ref{Facts} (2), it follows that $v_1$ is a true dividing vertex of either $T \oplus Y \oplus v_2$ (if $u_1 \not \in X_1$) or $T \oplus Y \oplus u_1 \oplus v_2$ (if $u_1 \in X_1$), again contradicting Claim \ref{dividingrelaxations}.
 \end{proof}
 
It follows from the definition of critical canvas that there exists an $(L,M)$-colouring of $C$ that does not extend to an $(L,M)$-colouring of $G$. For the remainder of the proof, let $\phi$ be one such fixed $(L,M)$-colouring of $C$. The following is an immediate consequence of Proposition \ref{CriticalSubgraph} and Claim \ref{NoProperCrit}.

\begin{claim}[Claim 5.20, \cite{postle2016five}]\label{colouringextends}
The colouring $\phi$ extends to every proper subgraph of $G$ that contains $C$ as a subgraph.  
\end{claim}

For $v \in V(G)\setminus V(C)$, we let $S(v) : = L(v) \setminus \{v[u, \phi(u)] \hskip 2mm | \hskip 2mm u\in N(v) \cap V(C)\}$. The following claim is analogous to Claim 5.21 in \cite{postle2016five}, with the appropriate changes for correspondence colouring.

\begin{claim}[Claim 5.21, \cite{postle2016five}]\label{s}
For all $v \in V(G) \setminus V(C)$, $|L(v)| = 5$ and $|S(v)| = 5-|N(v) \cap V(C)|$.
\end{claim}

%\begin{proof}
%Suppose for a contradiction that $|L(v)| \geq 6$ for some $v \in V(G)\setminus V(C)$. Let $c \in L(v)$, and let $L'$ be defined by $L'(v) : L(v) \setminus \{c\}$ and $L'(x) : = L(x)$ for all $x \in V(C) \setminus \{v\}$. Then $(G, C, (L', M))$ is a canvas, and $\phi$ clearly does not extend to an $(L',M)$-colouring of $G$. By Proposition \ref{CriticalSubgraph} the canvas $(G, C, (L', M)))$ has a critical subcanvas $(G', C, (L',M))$. Since $T$ was chosen to minimize $\sum_{v\in V(G)} |L(v)|$ subject to $|E(G)|$ being as small as possible, it follows that $G'$ is a proper subgraph of $G$. That contradicts Claim \ref{ProperCrit} applied to $T_0 = T$ and $G'$. 

%To prove the second statement, suppose for a contradiction that $|S(v)| \geq 5-|N(v) \cap V(C)|$ for some $v \in V(G) \setminus V(C)$. Thus by our choice of minimum counterexample, it follows that either $v$ has a neighbour $w_1$ such that $v[w_1, \phi(w_1)] = \emptyset$, or that $v$ has two distinct neighbours $w_1$ and $w_2$ in $V(C)$ such that $v[w_1, \phi(w_1)] = v[w_2, \phi(w_2)]$. But then $\phi$ does not extend to $G \setminus vw_1$, contradicting Claim \ref{colouringextends}.
%\end{proof}

\begin{figure}[ht]
\tikzset{white/.style={shape=circle,draw=black,fill=white,inner sep=1pt, minimum size=7pt}}
\tikzset{starnode/.style={inner sep=1pt, minimum size=7pt, star,star points=4,star point ratio=0.5, draw, fill=white}}
\tikzset{invisible/.style={shape=circle,draw=black,fill=black,inner sep=0pt, minimum size=0.1pt}}

\begin{center}
\hskip 6mm
\begin{tikzpicture}
        \node[invisible] (11) at (0,0){};
        \node[white] (1) at (1,0){};
        \node[white] (2) at (4,0){};
        \node[white] (3) at (5,0){};
        \node[white] (4) at (6,0){};
        \node[white] (5) at (8,0){};
        \node[] at (8,-0.4) {$u_2$};
        \node[white] (6) at (9,0){};
        \node[] at (9,-0.4) {$u_3$};
        \node[white] (7) at (2,2){};
        \node[] at (1.5,2) {$z_1$};
        \node[white] (8) at (5,2){};
        \node[] at (5.5,1.65) {$u_1$};
        \node[white] (9) at (7.6,2){};
        \node[] at (8,2) {$x_2$};
        \node[white] (10) at (5,4){};
        \node[] at (5,4.4) {$z_2$};
        \node[invisible] (12) at (10,0){};

        \draw[black] (11)--(1); 
        \draw[black] (1)--(2);
        \draw[black] (2)--(3);
        \draw[black] (3)--(4);   
        \draw[black] (4)--(5); 
        \draw[black] (5)--(6);
        \draw[black] (6)--(12);
        \draw[black] (1)--(7);
        \draw[black] (2)--(8);  
        \draw[black] (3)--(8);  
        \draw[black] (4)--(8);  
        \draw[black] (5)--(9);   
        \draw[black] (6)--(9);
        \draw[black] (7)--(8);  
        \draw[black] (8)--(9);  
        \draw[black] (7)--(10);  
        \draw[black] (8)--(10);  
        \draw[black] (9)--(10);  
        \draw[black] (2)--(8);

\end{tikzpicture}
\hskip 6mm
\begin{tikzpicture}
        \node[invisible] (11) at (0,0){};
        \node[white] (1) at (1,0){};
        \node[white] (2) at (4,0){};
        \node[white] (3) at (5,0){};
        \node[white] (4) at (6,0){};
        \node[white] (5) at (8,0){};
        \node[] at (8.1,-0.4) {$u_2$};
        \node[white] (6) at (9.5,2){};
        \node[] at (9.9,2.1) {$u_3$};
        \node[white] (13) at (9,0){};
        \node[white] (14) at (10,0){};
        \node[white] (15) at (11,0){};

        \node[white] (7) at (2,2){};
        \node[] at (1.5,2) {$z_1$};
        \node[white] (8) at (5,2){};
        \node[] at (5.5,1.65) {$u_1$};
        \node[white] (9) at (8,2){};
        \node[] at (8.1,2.3) {$x_2$};
        \node[white] (10) at (5,4){};
        \node[] at (5,4.4) {$z_2$};
        \node[invisible] (12) at (12,0){};

        \draw[black] (11)--(1); 
        \draw[black] (1)--(2);
        \draw[black] (2)--(3);
        \draw[black] (3)--(4);   
        \draw[black] (4)--(5); 
        \draw[black] (5)--(9);
        \draw[black] (5)--(13);
        \draw[black] (13)--(14);
        \draw[black] (14)--(15);
        \draw[black] (15)--(12);

        \draw[black] (6)--(13);
        \draw[black] (6)--(14);
        \draw[black] (6)--(15);

        \draw[black] (1)--(7);
        \draw[black] (2)--(8);  
        \draw[black] (3)--(8);  
        \draw[black] (4)--(8);  
        \draw[black] (5)--(9);   
        \draw[black] (6)--(9);
        \draw[black] (7)--(8);  
        \draw[black] (8)--(9);  
        \draw[black] (7)--(10);  
        \draw[black] (8)--(10);  
        \draw[black] (9)--(10);  
        \draw[black] (2)--(8);

\end{tikzpicture}
\caption{Vertices $z_1$, $z_2$ as described in Claim \ref{x2nbrpath}.} 
    \label{fig:z1z2}
\end{center}
\end{figure}
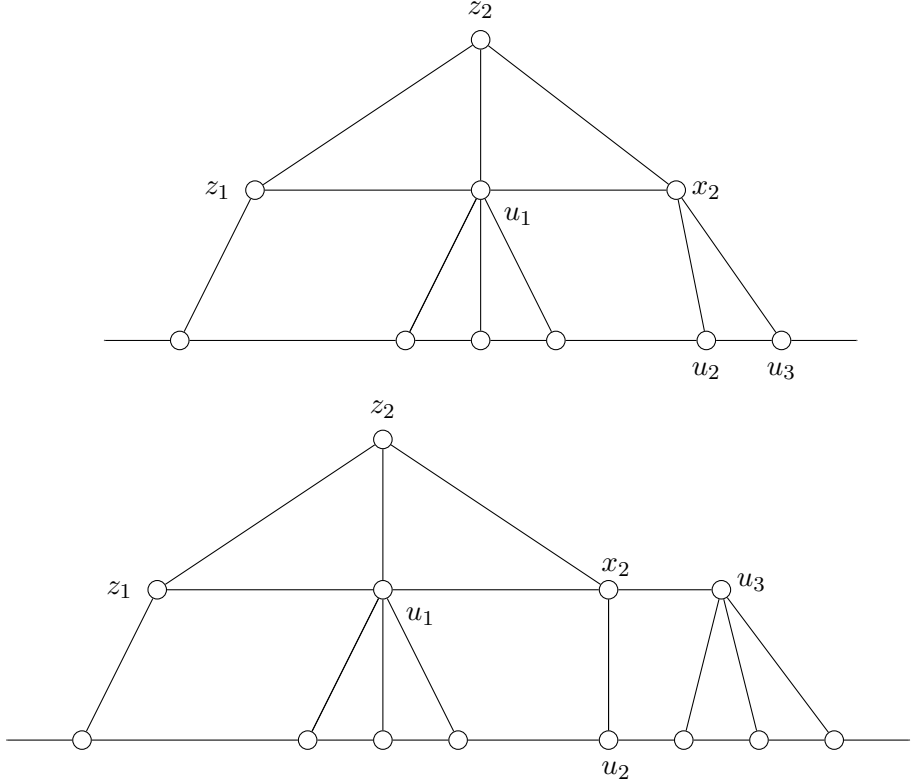

The following claim is nearly identical to Claim 5.22 in \cite{postle2016five}, and establishes some of the structure surrounding $X_1$ vertices that neighbour $X_2$ vertices. See Figure \ref{fig:z1z2} for an illustration of the vertices described in the claim.
\begin{claim}[Claim 5.22, \cite{postle2016five}]\label{degu}
Let $x_2 \in X_2$, and let $U = N(x_2) \cap X_1$. If $u \in U$, then $\deg(u) = 6$ and there exist adjacent vertices $z_1, z_2 \not \in V(C)$ such that $z_1$ is adjacent to $u$ and is in $B(T)$, and $z_2$ is adjacent to $u$ and $x_2$.
\end{claim}
\begin{proof}
By Claim \ref{QuasiSame} applied to $T$ and $u$, we find that $\deg(u) \leq 6$ and the claim follows unless $\deg(u) = 5$. So suppose for a contradiction that $\deg(u) = 5$.

Let $C' = (C \oplus U) \oplus x_2$, and $T \langle C' \rangle = (G', C', (L,M))$. Let $z \in V(G') \setminus V(C')$ be a neighbour of $u$. We claim that $G' \setminus uz$ has a $C'$-critical subgraph. To see this, we extend  $\phi$ to an $(L,M)$-colouring $\phi'$ of $G \setminus uz$ as follows. For $v \in V(C)$, let $\phi'(v) = \phi(v)$. Since $x_2 \not \in X_1$, we have that $|S(x_2)|\geq 3$, and $|S(u)|=2$ by Claim \ref{s} since $u \in X_1$ and every vertex in $X_1$ is a tripod of $T$ by Claim \ref{X1claims} (1). We may therefore choose $\phi'(x_2)$ with $u[x_2, \phi'(x_2)] \not \in S(u)$. Now if $\phi'$ extends to an $(L,M)$-colouring $\phi''$ of $G' \setminus uz$, then by re-defining $\phi''(u)$ to be a colour in {$S(u) \setminus u[z, \phi''(z)]$,} we obtain an extension of $\phi$ to $G$, a contradiction. Thus $\phi'$ does not extend to an $(L,M)$-colouring of $G\setminus uz$, and so by Proposition \ref{CriticalSubgraph} this proves our claim that $G' \setminus uz$ contains a $C'$-critical subgraph, say $G''$. But $G''$ is a proper $C'$-critical subgraph of $G'$, contradicting Claim \ref{criticalrelaxations}.
\end{proof}

\subsection{The Third Layer}\label{subsec:thirdlayer}
For the remainder of this section, let $X_3$ be the set of vertices $v \in V(G) \setminus (V(C) \cup X_1 \cup X_2)$ with at least three neighbours in $(C \oplus X_1) \oplus X_2$. We note that this set is not required and thus not used in \cite{postle2016five}; it is from this point forward that the proof of Theorem \ref{theorem:stronglinear} and its list colouring analogue diverge substantially. Having established that every vertex $u \in N(X_2) \cap X_1$ has degree 6, we are now equipped to prove the following claim.  The claim is analogous to Claim 5.19 in \cite{postle2016five} (or equivalently Claim \ref{X2claims} (1), above), with $X_3$ playing the role of $X_2$. Note that none of the proofs in \cite{postle2016five} require the introduction of this third layer of tri- and quadpods $X_3$.

\begin{claim}\label{X3claims}
$X_3 \neq \emptyset.$ Furthermore, let $x_3 \in X_3$, and let $u_1, u_2, \dots, u_k$ be all the neighbours of $x_3$ in $(C \oplus X_1) \oplus X_2$ listed in standard order. Then $k = 3$, and $u_2 \not \in X_2$.
\end{claim}
\begin{proof}
By Claim \ref{X2claims} (3), there does not exist a chord of $(C \oplus X_1) \oplus X_2$, and by Claim \ref{X2claims} (2), $(T \oplus X_1) \oplus X_2$ is a critical canvas. Thus by Theorem \ref{CycleChordTripod} it follows that $X_3 \neq \emptyset.$ Let $x_3 \in X_3$ and $u_1, \dots, u_k$ be as stated. Let $i \in \{2, \dots, k-1\}$. Suppose that $u_i \in X_2$. If both $u_{i-1}$ and $u_{i+1}$ are in $X_2$, then since there are no edges in $E(G)$ with both endpoints in $X_2$ by Claim \ref{X2claims}, it follows that $\deg(u_i) \leq 4$ as $u_i$ is adjacent to $x_3$ and $u_i$ has exactly three neighbours in $C \oplus X_1$ by Claim \ref{X2claims} (1). Thus at least one of $u_{i-1}$ and $u_{i+1}$ is not in $X_2$. If neither $u_{i-1}$ nor $u_{i+1}$ is in $X_2$, then given the existence of $u_i$ and the fact that $\deg(x_3) \geq 5$ by Proposition \ref{Facts} (2) it follows that $x_3$ is a true dividing vertex of $T \oplus (\{u_{i-1}, u_{i+1}\} \cap X_1)$, contradicting Claim \ref{dividingrelaxations}. Thus exactly one of $u_{i-1}$ and $u_{i+1}$ is in $X_2$. By symmetry, we may assume $u_{i-1} \in X_2$. Let $W = N(u_{i+1}) \cap X_1$. Note that $|W| \leq 2$ by Claim \ref{X2claims} (1). Then again we find that $x_3$ is a true dividing vertex: either of $T \oplus W \oplus u_{i-1} \oplus u_{i+1}$ (if $u_{i+1} \in X_1$) or of $T \oplus W \oplus u_{i-1}$ (if $u_{i+1} \in V(C)$), either way contradicting Claim \ref{dividingrelaxations}. 

We may assume that $k \geq 4$, as otherwise the claim holds. Since $x_3 \not \in X_2$, we may assume by symmetry that $u_1 \in X_2$. By above, $u_2, \dots, u_{k-1} \in V(C) \cap X_1$. Recall that by Claim \ref{X1claims} (3) and (1), there are no edges with both endpoints in $X_1$ and every member of $X_1$ has exactly three neighbours in $V(C)$. In addition, by Proposition \ref{Facts} (2) every vertex in $X_1$ has degree at least 5. By Claim \ref{degu}, if a vertex $u \in X_1$ is adjacent to a vertex in $X_2$, then it follows that $\deg(u) = 6$. Thus $\{u_2, u_3, \dots, u_{k-1}\} \subseteq V(C)$, and moreover, $|N(u_1) \cap X_1| = 1$. Similarly, if $u_k \in X_2$ then $|N(u_k)| = 1$. Then since $\deg(u_1) \geq 5$ and $\deg(u_k)\geq 5$ by Proposition \ref{Facts} (2) and every vertex of $X_2$ is a tripod of $C \oplus X_1$, it follows that $x_3$ is a true dividing vertex of either $(T \oplus (N(u_1) \cap X_1)) \oplus u_1$ (if $u_k \in V(C)$); or of $((T \oplus (N(u_1) \cap X_1)) \oplus u_1) \oplus u_k$ (if $u_k \in X_1$); or finally of $(((T \oplus (N(u_1) \cap X_1)) \oplus u_1) \oplus (N(u_k) \cap X_1)) \oplus u_k$ (if $u_k \in X_2$). Since as argued above $|N(u_1) \cap X_1| = 1$ and if $u_k \in X_2$ then $|N(u_k) \cap X_1| = 1$, it follows that $x_3$ is a true dividing vertex of a $k$-relaxation with $k \leq 4$, contradicting Claim \ref{dividingrelaxations}.
\end{proof}

Let $X_0:= V(C)$. It will be convenient to be able to succinctly describe the structure surrounding vertices in $X_2$ and $X_3$. To that end, we make the following definition.

\begin{definition}
Let $i \in \{2,3\}$, and let $x_i \in X_i$. Let $u_1, u_2, u_3, \dots$ be the neighbours of $x_i$ listed in standard counterclockwise order. We say $x_i$ is of type $(j, k, \ell)$ with $j, k, \ell \in \{0,1,2\}$ if $u_1 \in X_j$, $u_2 \in X_k$, and $u_3 \in X_\ell$.
\end{definition}

By Claim \ref{X3claims}, there exists a vertex $x_3 \in X_3$. For the remainder of the proof, we will fix such an $x_3$. Since $x_3 \not \in X_2$, it follows that $x_3$ has at least one neighbour in $X_2$.  By Claim \ref{X3claims}, if $x_3$ is of type $(j,k,\ell)$, then $k \neq 2$. Up to reflection of the graph, we may assume that $x_3$ is of one of the following types: (2,0,0), (2,0,1), (2,1,0), (2,1,1), (2,0,2), (2,1,2). The following claims allow us to rule out several of these types.

\begin{claim}\label{not201}
The vertex $x_3$ is not of type $(2,0,1)$.
\end{claim}
\begin{proof}
Suppose not. Let $x_2, x_0, x_1, \dots$ be the neighbours of $x_3$ in standard counterclockwise order. Note that by Claim \ref{X1claims} (1), $x_1$ is a tripod of $T$, and since $x_2 \not \in X_1$, it follows that $x_2$ has at most two neighbours in $V(C)$. Thus by Claim \ref{s}, $|S(x_2)| \geq 3$ and $|S(x_1)| = 2$. Moreover, since $x_3$ is adjacent to a vertex in $V(C)$, we have that $|S(x_3)| =4$. Thus there exists a colour $c_1 \in S(x_1)$ and a colour $c_2 \in S(x_2)$ such that $x_3[x_1,c_1] = x_3[x_1,c_2]$. Set $\phi(x_i) = c_i$ for $i \in \{1,2\}$. Let $U = N(x_2) \cap X_1$, and let $T' = T \langle ((C \oplus U) \oplus x_2) \oplus x_1 \rangle = (G', C', (L,M))$. We claim that $G'-x_3x_2$ has a $C'$-critical subgraph. To see this, note that if $\phi$ extends to an $(L,M)$-colouring of $G$, then by our choice of $\phi(x_1)$ and $\phi(x_2)$ this extension is also an $(L,M)$-colouring of $G$. Thus $\phi$ does not extend to an $(L,M)$-colouring of $G'-x_3x_2$, and so by Proposition \ref{CriticalSubgraph} we have that $G'-x_3x_2$ contains a $C'$-critical subgraph. But this subgraph is a proper subgraph of $G'$; and by Claim \ref{X2claims} (2), $|U| \leq 2$ and so $T'$ is a $k$-relaxation of $T$ with $k \leq 4$, contradicting Claim \ref{criticalrelaxations}.
\end{proof}

\begin{claim}\label{not210}
The vertex $x_3$ is not of type $(2,1,0)$.
\end{claim}
\begin{proof}
Suppose not. Let $x_2,x_1,x_0, \dots$ be the neighbours of $x_3$, listed in standard counterclockwise order. Note that $x_1$ has exactly three neighbours in $V(C)$ since $x_1$ is a tripod of $T$ by Claim \ref{X1claims} (1). If $x_2$ is adjacent to $x_1$, then since $G$ is planar and $x_1$ is adjacent to $x_3$, it follows that $\deg(x_1) = 5$. But this contradicts Claim \ref{degu} since $x_1x_2 \in E(G)$. Thus we may assume that $x_2$ is not adjacent to $x_1$, and so since $x_1$ has degree at least five by Proposition \ref{Facts} (2), there exists a vertex $x_1' \in X_1$ with $x_1'x_2 \in E(G)$ and $x_1'x_1 \in E(G)$. But this contradicts Claim \ref{X1claims} (3), since $G$ does not contain an edge with both endpoints in $X_1$.
\end{proof}

\begin{claim}\label{not211}
The vertex $x_3$ is not of type $(2,1,1)$.
\end{claim}
\begin{proof}
Suppose not. Let $x_2,x_1,x_1', \dots$ be the neighbours of $x_3$ listed in standard counterclockwise order. Note that $x_1$ has exactly three neighbours in $V(C)$ by Claim \ref{X2claims} (1), and that $x_1$ is not adjacent to $x_1'$ by Claim \ref{X1claims} (3). If $x_1$ is not adjacent to $x_2$, then since $x_1$ is adjacent to $x_3$ it follows that $\deg(x_1) \leq 4$, contradicting Proposition \ref{Facts} (2). Thus we may assume that $x_1$ is adjacent to $x_2$. But then $\deg(x_2) \leq 5$, contradicting Claim \ref{degu}.
\end{proof}
\begin{claim}\label{not202}
The vertex $x_3$ is not of type $(2,0,2)$.
\end{claim}
\begin{proof}
Suppose not. Let $x_2,x_0,x_2',\dots$ be the neighbours of $x_3$ listed in standard counterclockwise order. Note that $x_2$ is of type (1,0,0) and $x_2'$ is of type (0,0,1), as otherwise a vertex in $X_1$ adjacent to either $x_2$ or $x_2'$ has degree at most 4, contradicting Claim \ref{degu}. Let $x_1$ and $x_1'$ be the unique neighbours of $x_2$ and $x_2'$, respectively, in $X_1$.  By Claim \ref{s}, $|S(x_2)| = |S(x_2')| \geq 3$. Moreover, since $x_3$ is adjacent to a vertex in $V(C)$, we have that $|S(x_3)| =4$. Thus there exists an extension $\phi'$ of $\phi$ to $C \cup \{x_2, x_2'\}$ where $x_3[x_2, \phi'(x_2)] = x_3[x_2', \phi'(x_2')]$.  Let $T' = (G', C', (L,M)) = T\langle ((C \oplus \{x_1, x_1'\}) \oplus \{x_2, x_2'\}) \rangle$. We claim $G'-x_3x_2$ has a $C'$-critical subgraph. To see this, note that if $\phi'$ extends to an $(L,M)$-colouring of $G$, then by our choice of $\phi(x_2)$ and $\phi(x_2')$ this extension is also an $(L,M)$-colouring of $G$ that extends $\phi$. Thus $\phi'$ does not extend to an $(L,M)$-colouring of $G'-x_3x_2$, and so by Proposition 2.9 $G'-x_3x_2$ contains a $C'$-critical subgraph. But this subgraph is a proper subgraph of $G'$, and $G'$ is a $4$-relaxation of $G$, contradicting Claim \ref{criticalrelaxations}.
\end{proof}

\begin{claim}\label{not212}
The vertex $x_3$ is not of type $(2,1,2)$.
\end{claim}
\begin{proof}
Suppose not. Let $x_2, x_1, x_2', \dots$ be the neighbours of $x_3$ listed in standard counterclockwise order. By Claim \ref{X1claims} (1), $x_1$ has exactly three neighbours in $V(C)$. By Proposition \ref{Facts} (2), $x_1$ is adjacent to at least one of $x_2$ and $x_2'$. But then by Claim \ref{degu}, it follows that $x_1$ is adjacent to both $x_2$ and $x_2'$. By Claim \ref{s}, $|S(x_2)| \geq 3$, $|S(x_2')| \geq 3$, and $|S(x_1)| = 2$. Thus there exists an extension $\phi'$ of $\phi$ to $C \cup \{x_2, x_2'\}$ where $x_1[x_2,\phi'(x_2)] \not \in S(x_1)$ and similarly $x_1[x_2',\phi'(x_2')] \not \in S(x_1)$. Let $U = (N(x_2) \cup N(x_2')) \cap X_1$, and let $T' = (G', C', (L,M)) = T\langle C \oplus U \oplus \{x_2, x_2'\} \rangle$. We claim $G'-x_3x_1$ has a $C'$-critical subgraph. To see this, note that if $\phi'$ extends to an $(L,M)$-colouring of $G$, then redefining $\phi'(x_1) \in S(x_1) \setminus \phi'(x_3)$ we obtain an $(L,M)$-colouring of $G$ that extends $\phi$. By Proposition \ref{CriticalSubgraph} $G'-x_3x_1$ contains a $C'$-critical subgraph. But this subgraph is a proper subgraph of $G'$, and $G'$ is a $k$-relaxation of $G$ with $k \leq 5$ (since $x_1 \in N(x_2) \cap N(x_2')$), contradicting Claim \ref{criticalrelaxations}.
\end{proof}

%   Okay, we've established that x_3 is of type (2,0,0). From there...

The following claim follows immediately from Claims \ref{not201}-\ref{not212}.

\begin{claim}\label{x3type}
    The vertex $x_3$ is a vertex of type $(2,0,0)$.
\end{claim}

For the remainder of the proof, let $x_2$ be the neighbour of $x_3$ in $X_2$.

\begin{claim}\label{x2type}
The vertex $x_2$ is of type $(1,0,0)$. 
\end{claim}
\begin{proof}
Suppose not. Note that since $x_2 \not \in X_1$, we have that $x_2$ is adjacent to at least one vertex in $X_1$. By Claim \ref{X2claims} (2), since $x_2$ is not of type $(1,0,0)$, it follows that $x_2$ is of type $(1,0,1)$. Let $x_1$ and $x_1'$ be the neighbours of $x_2$ in $X_1$. Since $G$ is planar, one of $x_1$ and $x_1'$ has degree at most 5, since it is adjacent to three vertices in $V(C)$ by Claim \ref{X1claims} (1); to $x_2$; and possibly to $x_3$. This contradicts Claim \ref{degu}.
\end{proof}

For the remainder of the proof, let $x_1$ be the neighbour of $x_2$ in $X_1$, let $T_1 = T \oplus x_1$, and let $T_2 = T_1 \oplus x_2$. The following claim is similar to Claim \ref{QuasiSame}, and establishes some of the structure of $T_2$. 
\begin{claim}\label{QuasiQuasiSame}
Either:
\begin{itemize}
    \item $\deg(x_2) = 5$, or 
    \item $\deg(x_2) = 6$, the neighbours of $x_2$ not in $X_1$ or $C$ form a path of length two with the ends in $B(T \oplus x_1)$, $b(T_2) = b(T)$, $q(T_2) = q(T)$, and $d(T) \geq d(T_2)-2\varepsilon$. 
\end{itemize}
\end{claim}
\begin{proof}
Suppose not. Note that we may assume $\deg(x_2) \geq 6$, since every vertex in $V(G)\setminus V(C)$ has degree at least five by Proposition \ref{Facts} (2) and $\deg(x_2) \neq 5$ by assumption. Moreover, by Claim \ref{v(T)} we have that $v(T) \geq 8$ and so $v(T_2) \geq 6$. By Corollary \ref{SubCycle}, $T_2$ is a critical canvas, and so by the minimality of $T$ it follows that  $d(T_2) \geq 3-\gamma$. In addition, since $x_1$ is a tripod of $T$ and $x_2$ is a tripod of $T \oplus x_1$, we have that $\defc(T) = \defc(T_2)$, and $v(T) = v(T_2)+2$. It follows that
$$
d(T) = d(T_2) -2\varepsilon+\alpha(b(T_2)-b(T)+q(T_2)-q(T)).
$$
By Claim \ref{degu}, $\deg(x_1) = 6$ and there exist adjacent vertices $z_1,z_2 \not \in V(C)$ such that $z_1$ is adjacent to $x_1$ and is in $B(T)$, and $z_2$ is adjacent to $x_1$ and $x_2$. Let $x_1,u_2,u_3,q_1,\cdots, z_2$ be the neighbours of $x_2$ listed in their cyclic order around $x_2$. Let $R = N(x_2) \setminus \{x_1,u_2,u_3,q_1,z_2\}$.  We claim that $R \cap Q(T_1) = \emptyset$. To see this, suppose not: let $q$ be a vertex in $R \cap Q(T_1)$. Then given the presence of the path $z_1z_2x_2$ and the fact that $z_1 \in B(T)$, it follows that $q$ is cofacial with a vertex in $V(C)$. Given the existence of $q_1$, this implies that $q$ is a strong dividing vertex of $T_2$. Since $T_2$ is a 2-relaxation of $T$, this contradicts Claim \ref{dividingrelaxations}. Thus $R \cap Q(T_1) = \emptyset$; and since $B(T_2) \subseteq Q(T_2)$, it follows that $R \cap B(T_1) = \emptyset$.
Since $\deg(x_1) = 6$ by Claim \ref{degu}, Claim \ref{QuasiSame} implies that $b(T) = b(T_1)$, $q(T) = q(T_1),$ and $d(T) \geq d(T_1)-\varepsilon$. Since $R \cap Q(T_2) = \emptyset$, it follows that $q(T_2) \geq q(T_1) + |R|-1$, and $b(T_2) \geq b(T_1) + |R|-1$. Moreover, $\defc(T_2) = \defc(T_1)$, and $v(T_1) = v(T_2) + 1$. Thus
\begin{equation}\label{eq:dT_1}
    d(T_1) = d(T_2) -\varepsilon + \alpha(b(T_2)-b(T_1)+q(T_2)-q(T_1)),
\end{equation}
and so
\begin{align*}
    d(T_1) &\geq d(T_2)-\varepsilon + \alpha(b(T_1) + |R|-1-b(T_1)+q(T_1)+|R|-1-q(T_1)) \\
    &\geq d(T_2)-\varepsilon+\alpha(2|R|-2).
\end{align*}
Since $d(T) \geq d(T_1)-\varepsilon$ by Claim \ref{QuasiSame}, it follows that 
\begin{align*}
    d(T) \geq d(T_2)-2\varepsilon + \alpha(2|R|-2) 
\end{align*}
If $|R| \geq 2$, then $d(T) \geq d(T_2)-2\varepsilon+2\alpha$. This is a contradiction, since by the minimality of $T$ we have that $d(T_2) \geq 3-\gamma$, and $\varepsilon < \alpha$ by (I1), implying that $d(T) \geq 3-\gamma$. Thus $|R| = 1$, and so $\deg(x_2) = 6$, $q(T_2) \geq q(T_1)$, and $b(T_2) \geq b(T_1)$. If equality does not hold in both of these expressions, then in Equation (\ref{eq:dT_1}) we have
\begin{align*}
    d(T_1) &> d(T_2)-\varepsilon + \alpha \\
    d(T) \geq d(T_1)-\varepsilon &> d(T_2)-2\varepsilon + \alpha \textnormal{\hskip 2mm since $d(T) \geq d(T_1)-\varepsilon$ by Claim \ref{QuasiSame}}\\
    &> 3-\gamma-2\varepsilon+\alpha,
\end{align*}
where the last line follows from the fact that by the minimality of $T$, we have $d(T_2) \geq 3-\gamma$. This is a contradiction, since $\alpha \geq 2\varepsilon$ by (I1). Hence $q(T_1) = q(T_2)$, and $b(T_1) = b(T_2)$. Note this implies $b(T) = b(T_2)$ and $q(T) = q(T_2)$ by Claims \ref{degu} and \ref{QuasiSame}. Moreover, $d(T_2) \geq d(T_1)-\varepsilon$; and since $d(T) \geq d(T_1) -\varepsilon$ by Claim \ref{QuasiSame}, it follows that $d(T) \geq d(T_2)-2\varepsilon$.

Now, let $q \in R$. From the above, $Q(T_2)\setminus \{q\} = Q(T_1)\setminus \{x_2\}$, which implies that $q_1qq_2$ forms a path as otherwise there would exist a vertex $x \not \in \{q, q_1, q_2\}$ with $x \in Q(T_2) \cap Q(T_1)$. But then $x$ is a strong dividing vertex of $T_1$, contradicting Claim \ref{dividingrelaxations}. Similarly, $B(T_2)\setminus \{q\} = B(T_1) \setminus \{x_2\}$, implying that $\{q_1, q_2\} \subseteq B(T_1)$.
\end{proof}

\begin{figure}[ht]
\tikzset{white/.style={shape=circle,draw=black,fill=white,inner sep=1pt, minimum size=7pt}}
\tikzset{starnode/.style={inner sep=1pt, minimum size=7pt, star,star points=4,star point ratio=0.5, draw, fill=white}}
\tikzset{invisible/.style={shape=circle,draw=black,fill=black,inner sep=0pt, minimum size=0.1pt}}

\begin{center}
\hskip 6mm

\begin{tikzpicture}
        \node[white] (1) at (1,0){};
        \node[white] (2) at (3,0){};
        \node[white] (3) at (4,0){};
        \node[white] (4) at (5,0){};
        \node[white] (5) at (7,0){};
        \node[white] (6) at (8,0){};
        \node[white] (7) at (10,0){};
        \node[white] (8) at (11,0){};
        \node[white] (9) at (1.5,2){};
        \node[] at (1.1,2.1){$z_1$};
        \node[white] (10) at (4,2){};
        \node[] at (4.3,2.3){$x_1$};
        \node[white] (11) at (7,2){};
        \node[] at (7.5,2.2){$x_2$};
        \node[white] (12) at (10,2){};
        \node[] at (10.3,2.2){$x_3$};
        \node[white] (13) at (4,4){};
        \node[] at (4,4.3){$z_2$};
        \node[white] (14) at (8.5,4){};
        \node[] at (8.5,4.3){$z_3$};
       \node[invisible] (15) at (0,0){};
        \node[invisible] (16) at (12,0){};

        \draw[black] (1)--(2);
        \draw[black] (2)--(3);
        \draw[black] (3)--(4);   
        \draw[black] (4)--(5); 
        \draw[black] (5)--(6); 
        \draw[black] (6)--(7); 
        \draw[black] (7)--(8); 
        \draw[black] (9)--(10);
        \draw[black] (10)--(11);
        \draw[black] (11)--(12);
        \draw[black] (13)--(14);

        \draw[black] (1)--(9);
        \draw[black] (2)--(10);
        \draw[black] (3)--(10);
        \draw[black] (4)--(10);
        \draw[black] (5)--(11);
        \draw[black] (6)--(11);
        \draw[black] (7)--(12);
        \draw[black] (8)--(12);

        \draw[black] (9)--(13);
        \draw[black] (10)--(13);
        \draw[black] (11)--(13);
        \draw[black] (11)--(14);
        \draw[black] (12)--(14);
        \draw[black] (15)--(1);  
        \draw[black] (16)--(8);  
        
\end{tikzpicture}
\caption{Vertices $z_1$, $z_2$, and $z_3$ as described in Claims \ref{x2nbrpath} and \ref{degx3}. For each $i \in \{1,2,3\}$, the vertex $x_i$ is in $X_i$. Moreover, $x_2$ is of type (1,0,0), and $x_3$ is of type (2,0,0).}
    \label{fig:z1z2z3}
\end{center}
\end{figure}
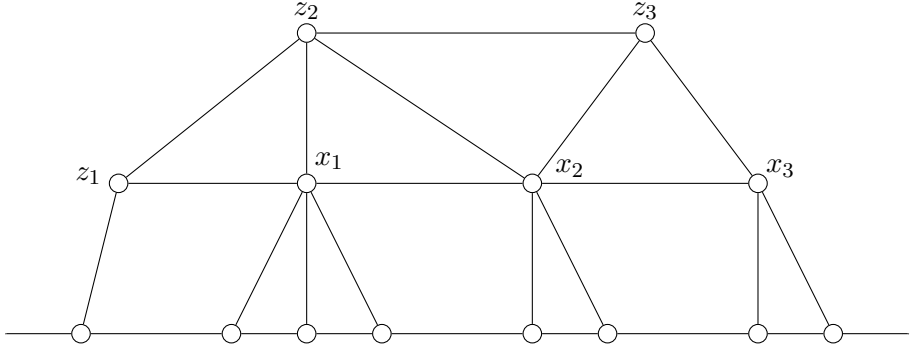

The following claim establishes that $\deg(x_2) = 6$. See Figure \ref{fig:z1z2z3} for an illustration of the vertices described in the claim. For the remainder of the proof, let $T_3 = T_2 \oplus x_3$, with $T_3 = (G_3, C_3, (L,M))$.

\begin{claim}\label{x2nbrpath}
The following both hold.
\begin{itemize}
    \item $\deg(x_2) = 6$, and
    \item there exist adjacent vertices $z_2, z_3 \not \in V(C)$ such that $z_2$ is adjacent to $x_2$ and $x_1$, and $z_3$ is adjacent to $x_2$ and $x_3$.
\end{itemize} 
\end{claim}
\begin{proof}
By Claim \ref{QuasiQuasiSame}, this holds unless $\deg(x_2) = 5$.  Let $z \in V(G_3) \setminus V(C_3)$ be a neighbour of $x_2$. We claim $G \setminus \{x_2z\}$ has a $C_3$-critical subgraph. To see this, we start by extending $\phi$ to a partial $(L,M)$-colouring of $C \cup C_3$ as follows: first, note that since $x_1 \in X_1$ is a tripod of $T$ by Claim \ref{X1claims} (1), Claim \ref{s} implies that $|S(x_1)| = 2$. Similarly, since $x_2$ is of type (1,0,0) by Claim \ref{x2type}, we have that Claim \ref{s} implies that $|S(x_2)| = 3$. Since $x_3$ is of type (2,0,0), again Claim \ref{s} implies that $|S(x_3)| = 3$. By Claim \ref{matchmin}, $|M_{x_1x_2}| = |M_{x_2x_3}| = 5$. Thus there exists a colour $c_1 \in S(x_1)$ and a colour $c_3 \in S(x_3)$ such that $x_2[x_1, c_1] = x_2[x_3, c_3]$. Set $\phi(x_i) = c_i$ for $i \in \{1,3\}$. If $\phi$ extends to an $(L,M)$-colouring $\phi'$ of $G_3\setminus x_2z$, then $\phi'$ extends to an $(L,M)$-colouring of $G$ by redefining $\phi'(x_2)$ to be a colour in $S(x_2) \setminus \{x_2[x_1, c_1], x_2[z, \phi'(z)]\}$. Since $|S(x_2)| = 3$, such a choice exists \textemdash a contradiction, since $\phi$ does not extend to $G$. Thus $\phi$ does not extend to an $(L,M)$-colouring of $G_3\setminus x_2z$, and thus by Proposition \ref{CriticalSubgraph}, we have that $G_3\setminus x_2z$ has a $C_3$-critical subgraph, $G_3'$. But then $G_3'$ is a proper $C_3$-critical subgraph of $G_3$ (since $x_2z \not \in E(G_3')$). Since $T_3$ is a $3$-relaxation of $T$, this contradicts Claim \ref{criticalrelaxations}.
\end{proof}

For the remainder of the proof, let $z_1,z_2,$ and $z_3$ be as in Claims \ref{degu} and \ref{x2nbrpath}.
\begin{claim}\label{z2z3vc}
Neither $z_2$ nor $z_3$ have a neighbour in $V(C)$.
\end{claim}
\begin{proof}
Note that $N(z_2) \cap V(C)= \emptyset$, as otherwise given the existence of $z_1$ and $z_3$, we have that $z_2$ is a strong dividing vertex of $T_1$, contradicting Claim \ref{dividingrelaxations}. Moreover, note that $\deg(x_3) \geq 5$ by Proposition \ref{Facts} (2); and so, given that $x_3$ is a tripod of $T_2$ by Claim \ref{X3claims}, there exists a vertex $q_1 \in N(x_3)$ such that $q_1,z_3,x_2$ is part of the cyclic ordering of the neighbours of $x_3$ and $q_1 \not \in V(C)$. Thus $N(z_3) \cap V(C) = \emptyset$ by Claim \ref{z2z3vc}, as otherwise given the existence of $z_2$ and $q_1$, we have that $z_3$ is a strong dividing vertex of $T_3$, contradicting Claim \ref{dividingrelaxations}.
\end{proof}
The following claim bounds $d(T)$ in terms of $d(T_3)$ and establishes that $\deg(x_3) = 6$.

\begin{claim}\label{degx3}
$\deg(x_3) = 6$ and $d(T) \geq d(T_3)-3\varepsilon$. 
\end{claim}
\begin{proof}
Suppose not. First suppose that $\deg(x_3) \geq 6$. Note that since $x_1$ is a tripod of $T$; since $x_2$ is a tripod of $T_1$; and since $x_3$ is a tripod of $T_2$, it follows that $\defc(T) = \defc(T_1) = \defc(T_2) = \defc(T_3)$. Moreover, by Claim \ref{v(T)}, $v(T_3) \geq 5$. Thus by the minimality of $T$, we have that $d(T_3) \geq 3-\gamma$. In addition, note that $v(T_3) = v(T_2)-1$, that $v(T_2) = v(T_1)-1$, and that $v(T_1) = v(T)-1$. Thus, letting $T_0 := T$, we have that
\begin{equation}\label{dtbound}
       d(T_{i-1}) = d(T_{i}) - \varepsilon + \alpha\left(b(T_{i})-b(T_{i-1})+q(T_i)-q(T_{i-1}) \right) 
\end{equation}
for each $i \in \{1,2,3\}$.

Let $z_3, x_2, u_1, u_2, q_1, q_2, \dots $ be the neighbours of $z_3$ listed in their cyclic order around $z_3$, where $\{u_1, u_2 \} \subseteq V(C)$. Let $R = N(x_3) \setminus \{z_3,x_2,u_1,u_2,q_1\}$. We claim no vertex  $r \in R$ is in the quasiboundary of $T_2$; otherwise, given the existence of $z_3$ and $q_1$, we have that $r$ is a strong dividing vertex of $T_3$, contradicting Claim \ref{dividingrelaxations}. 

Note that $R \subseteq B(T_3) \subseteq Q(T_3)$. Thus since $x_3 \in Q(T_2) \setminus Q(T_3)$ and $R \subseteq Q(T_3) \setminus Q(T_2)$, it follows from Equation \ref{dtbound} that $$d(T_2) \geq d(T_3)-\varepsilon + 2\alpha(|R|-1).$$ By Claims \ref{QuasiQuasiSame} and \ref{x2nbrpath}, $d(T) \geq d(T_2)-2\varepsilon$. 

Combining these results, we have that 
\begin{align*}
    d(T) &\geq d(T_2)-2\varepsilon \\
    &\geq (d(T_3)-\varepsilon + 2\alpha(|R|-1))-2\varepsilon.
\end{align*}
Suppose that $\deg(x_3) \geq 7$, and so that $|R| \geq 2$. Then  $d(T) \geq d(T_3)-3\varepsilon + 2\alpha$. Since $d(T_3) \geq 3-\gamma$, this implies $d(T) \geq 3-\gamma-3\varepsilon+2\alpha$. This is a contradiction, since $3\varepsilon \leq 2\alpha$ by (I1). Thus $|R| = 1$, and so $\deg(x_3)=6$ and $d(T) \geq d(T_3)-3\varepsilon$, as desired.

Suppose now that $\deg(x_3) < 6$, and so that $\deg(x_3) = 5$ by Proposition \ref{Facts} (2). Recall that by Claim \ref{X2claims} (1), $x_1$ is a tripod of $T$. It follows from Claim \ref{s} that $|S(x_1)| = 2$; and moreover since $x_2$ is a tripod of $T_1$ by Claim \ref{X2claims} (1), we have further that $|S(x_2)| =3$.  Thus there exists a colour $c_2 \in S(x_2)$ such that $x_1[x_2, c_2] \not \in S(x_1)$. Let $\phi(x_2) = c_2$. Note that $N(z_2) \cap V(C)= \emptyset$ by Claim \ref{z2z3vc}. Thus by Claim \ref{s}, $|S(x_2)| =5$, and so there exists two distinct colours $d_1$ and $d_2$  in $S(z_2)$ such that $x_1[z_2, d_i] \not \in S(x_1)$ and $x_2[z_2, d_i] \neq c_2$ for $i \in \{1,2\}$. 

Furthermore, $N(z_3) \cap V(C) = \emptyset$ by Claim \ref{z2z3vc}. Thus by Claim \ref{s}, $|S(z_3)| = 5$. It follows that there exists an $i \in \{1,2\}$ such that $S(z_3) \setminus (\{z_3[z_2, d_i], z_3[x_2, c_2]\} \cup \{z_3[x_3, c]: c \in S(x_3)\})$ is non-empty, since $|S(z_3)| = 5$ and $z_3[z_2, d_1] \neq z_3[z_2, d_2]$. Without loss of generality, suppose that $i = 1$.  Let $d_3 \in S(z_3) \setminus (\{z_3[z_2, d_i], z_3[x_2, c_2]\} \cup \{z_3(x_3, c): c \in S(x_3)\})$. Finally, let $\phi(z_2) = d_1$ and $\phi(z_3) = d_3$.

Let $C'$ be obtained from $C_3$ by deleting $x_2$ and adding the vertices $z_2,z_3$ as well as edges $x_1z_2, z_2z_3,$ and $z_3x_3$. Let $T' = T\langle C' \rangle =  (G', C', (L,M))$. We claim that $G'-\{x_3q_1, x_1z_1\}$ has a $C'$-critical subgraph. To see this, note that if $\phi$ extends to $G'-\{x_3q_1, x_1z_1\}$, then $\phi$ extends to an $(L,M)$-colouring of $G$ by redefining $\phi(x_1)$ as a colour in $S(x_1) \setminus x_1[z_1, \phi(z_1)]$ and $\phi(x_3)$ as a colour in $S(x_3) \setminus \{x_3[q_1, \phi(q_1)], x_3[x_2, c_2]\}$. Such choices exist, since $|S(x_1)| = 2$ and $|S(x_3)|= 3$ by Claim \ref{s} since $x_3$ has exactly two neighbours in $V(C)$ by Claims \ref{not201}-\ref{not212}. This contradicts the fact that $\phi$ does not extend to $G$.  Thus $\phi$ does not extend to an $(L,M)$-colouring of $G'-\{x_3q_1, x_1z_1\}$, and so by Proposition \ref{CriticalSubgraph} we have that $G'-\{x_3q_1, x_1z_1\}$ has a $C'$-critical subgraph $G''$. Note that $v(T') \geq 2$ by Claim \ref{v(T)}, and $|E(G') \setminus E(G'')| \geq 2$.  Moreover, we claim that $C'$ is chordless: this follows easily from the facts that $C$ is chordless by Claim \ref{Chord}; that neither $z_2$ nor $z_3$ have a neighbour in $V(C)$ by Claim \ref{z2z3vc}; that $x_1z_3 \not \in E(G)$ since $G$ is planar; and that $z_2x_3 \not \in E(G)$ since $z_3$ has degree at least five by Proposition \ref{Facts} (2). Thus $|E(G'') \setminus E(C')| \geq 2$.

By Claim \ref{ProperCrit} (3)  applied to $T'$ and $G''$, we find that $d(T') \geq 5-(2\alpha + \varepsilon) - \gamma$. Moreover, $v(T') = v(T_2)+3$, $b(T') \geq b(T_2)-3$, and similarly $q(T') \geq q(T_2)-3$. Thus $s(T') \geq s(T_2)-(3\varepsilon +6\alpha)$. Furthermore, $\defc(T') = \defc(T_2)+1$. Putting all of this together,
\begin{align*}
    5-(2\alpha + \varepsilon)-\gamma &\leq d(T') \\
    &\leq \defc(T')-s(T') \\
    &\leq (\defc(T_2)+1)-(s(T_2)-(3\varepsilon + 6\alpha)) \\
    &\leq d(T_2)+1+(3\varepsilon + 6\alpha),
\end{align*}
which implies that $d(T_2) \geq 4-\gamma -(8\alpha+4\varepsilon)$. By Claim \ref{QuasiQuasiSame}, $d(T) \geq d(T_2)-2\varepsilon$, and so $d(T) \geq 4-\gamma-(8\alpha +6\varepsilon)$. This is a contradiction, since $8\alpha+6\varepsilon \leq 1$ by (I2) and (I3).
\end{proof}

We will complete the proof of Theorem \ref{theorem:stronglinear} by showing that $x_3$ is not of type (2,0,0), thereby arriving at a contradiction. Before we do this, we need the following key claim which establishes some of the correspondence assignment in the graph near $x_3$.

\begin{claim}\label{keyclaim}
The following both hold. 
\begin{enumerate}
    \item [(i)] The vertex $z_1$ has exactly one neighbour in $V(C)$, and there do not exist colours $d_1 \in S(x_1)$ and $d_2 \in S(x_2)$ such that $z_2[x_1, d_1] = z_2[x_2, d_2]$.
    \item [(ii)] Let $T_1 = (G_1, C_1, (L,M))$.  Let $\phi(x_1) \in S(x_1)$, let $S'(v) = S(v)$ for all $v \in V(G') \setminus N(x_1)$, let $S'(v) = S(v) \setminus v[x_1, \phi(x_1)]$ for each $v \in N(x_1)$. The vertex $z_2$ has exactly one neighbour in $V(C')$. Moreover, there do not exist colours $d_2 \in S'(x_2)$ and $d_3 \in S'(x_3)$ such that $z_3[x_2, d_2] = z_3[x_3, d_3]$.
\end{enumerate}
\end{claim}

\begin{proof}
We begin by proving (i). First we will show that $z_1$ has exactly one neighbour in $V(C)$. To see this, suppose not. Since $z_1$ is in the boundary of $T$ by Claim \ref{degu}, it follows that $z_1$ has at least one neighbour in $V(C)$. Thus $z_1$ has at least two neighbours in $V(C)$. Since $x_1$ is adjacent to $z_1$ and there are no edges in $E(G)$ with both endpoints in $X_1$ by Claim \ref{X1claims} (3), it follows that that $z_1 \not \in X_1$ and so that $z_1$ has exactly two neighbours in $V(C)$. Thus by Claim \ref{s} we have that $|S(z_1)| = 3$. Similarly, since $x_2$ is a tripod of $T_1$ by Claim \ref{X2claims} (1), by Claim \ref{s}, $|S(x_2)| = 3$. Since $x_1$ is a tripod of $T$ by Claim \ref{X1claims} (1), we have further from Claim \ref{s} that $|S(x_1)| = 2$. Thus there exists a colour $c_1 \in S(z_1)$ and a colour $c_2 \in S(x_2)$ such that $x_1[z_1,c_1] \not \in S(x_1)$ and $x_1[x_2,c_2] \not \in S(x_1)$. Let $C' = C \oplus x_1 \oplus z_1 \oplus x_2$, and let $\phi(z_1) = c_1$ and $\phi(x_2) = c_2$. Let $T' = (G', C', (L,M)) = T\langle C' \rangle$. We claim $G' - x_1z_2$ has a $C'$-critical subgraph. To see this, note that if $\phi$ extends to an $(L,M)$-colouring of $G'-x_1z_2$, then by redefining $\phi(x_1)$ to be a colour in $S(x_1) \setminus x_1[z_2, \phi(z_2)]$ we obtain an extension of $\phi$ to an $(L,M)$-colouring of $G$, a contradiction. Thus $\phi$ does not extend to $G'-x_1z_2$, and so by Proposition \ref{CriticalSubgraph}, we have that $G'-x_1z_2$ contains a $C'$-critical subgraph $G''$. But $G''$ is a proper subgraph of $G'$, and $T'$ is a 3-relaxation of $T$. This contradicts Claim \ref{criticalrelaxations}.

Thus $z_1$ has exactly one neighbour in $V(C)$, and so by Claim \ref{s} we have that $|S(z_1)| = 4$. Since $|S(x_1)| =2$, there exist two distinct colours $c_1$ and $c_2$ in $S(z_1)$ with $x_1[z_1, c_1] \not \in S(x_1)$ and $x_1[z_1,c_2] \not \in S(x_1)$.

\begin{figure}[ht]
\tikzset{black/.style={shape=circle,draw=black,fill=black,inner sep=1pt, minimum size=9pt}}
\tikzset{white/.style={shape=circle,draw=black,fill=white,inner sep=1pt, minimum size=9pt}}

\tikzset{invisible/.style={shape=circle,draw=black,fill=black,inner sep=0pt, minimum size=0.1pt}}
\begin{center}
\begin{tikzpicture}
\filldraw[color=black!100, fill=black!0, very thick] (0,-3) ellipse (1.7 and 0.7);
\filldraw[color=black!100, fill=black!0, very thick] (5,0) ellipse (0.7 and 1.7);
\filldraw[color=black!100, fill=black!0, very thick] (-5,0) ellipse (0.7 and 1.7);
\filldraw[color=black!100, fill=black!0, very thick] (0,3) ellipse (1.7 and 0.7);
        \node[] at (-6.5,0) {$S(z_1)$};
        \node[] at (2.5,-3) {$S(x_1)$};
        \node[] at (6.5,0) {$S(x_2)$};
        \node[] at (0,4) {$S(z_2)$};
        
        \node[] at (0,0) {$M_{x_1z_2}$};
        \node[] at (-3.6,2.2) {$M_{z_1z_2}$};
        \node[] at (3.5,2) {$M_{z_2x_2}$};
        \node[] at (-3,-2.5) {$M_{z_1x_1}$};

        \node[black] (1) at (-5,0.9){};
        \node[white] (2) at (-5,0.3){$c_i$};
        \node[black] (3) at (-5,-0.3){};
        \node[black] (4) at (-5,-1){};
        \node[black] (5) at (-0.75,-3){};
        \node[black] (6) at (0.75,-3){};
        \node[black] (7) at (5,-0.5){};
        \node[black] (8) at (5,0){};
        \node[black] (9) at (5,0.5){};
        \node[black] (10) at (1,3){};
        \node[black] (11) at (0.5,3){};
        \node[black] (12) at (0,3){};
        \node[black] (13) at (-0.5,3){};
        \node[white] (14) at (-1.1,3){$c_3$};

        \draw[black] (3)--(6); 
        \draw[black] (4)--(5); 
        \draw[black] (6)--(10); 
        \draw[black] (9)--(10); 
        \draw[black] (8)--(11); 
        \draw[black] (7)--(12); 
        \draw[dashed] (5)--(13);
        \draw[dashed] (2)--(13);
        \draw[dashed] (1)--(14);

\end{tikzpicture}

\caption{The matchings between $S(x_1)$, $S(x_2)$, $S(z_1)$, and $S(z_2)$, as described in Claim \ref{keyclaim}. The matching $M_{x_2x_1}$ is omitted for clarity. We assume there exists a colour $d_1 \in S(x_1)$ and $d_2 \in S(x_2)$ such that $z_2[x_1,d_1] = z_2[x_2,d_2]$. Without loss of generality, we may assume that the solid edges in the matchings are as shown. No matter the remainder of the edges in $M_{z_1z_2}$ and $M_{x_1z_2}$, there exist colours $c_i \in S(z_1)$ and $c_3 \in S(z_2)$ such that $c_i$ is unmatched in $M_{x_1z_1}$, such that $c_3$ is unmatched in $M_{x_1z_2}$ and $M_{x_2z_2}$, and such that $c_3 \neq z_2[z_1, c_i]$.} 
    \label{fig:matchings1}
\end{center}
\end{figure}
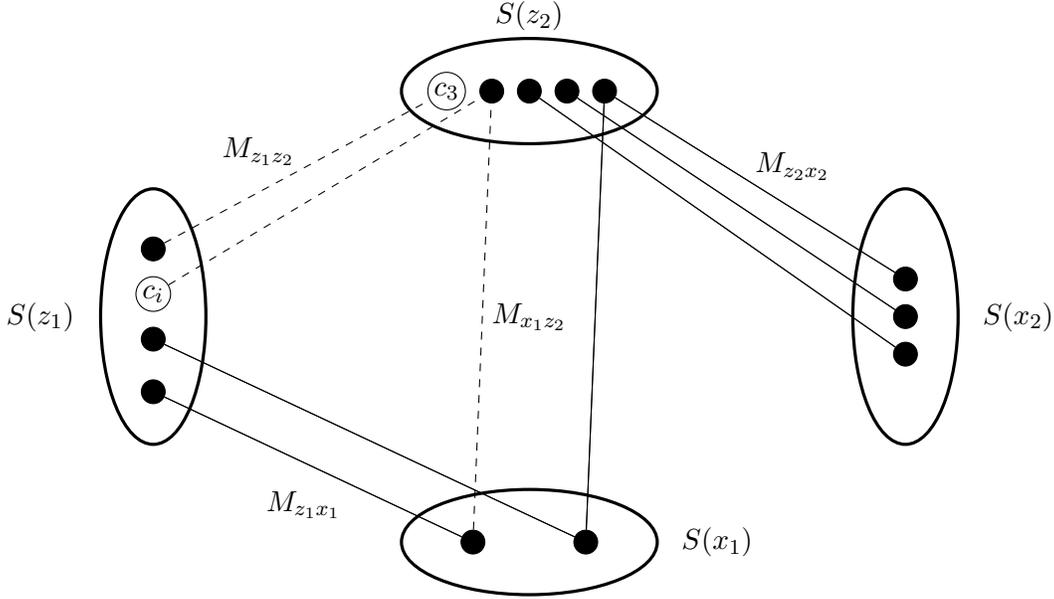

We now proceed with the rest of the claim. Suppose for a contradiction that there exist $d_1 \in S(x_1)$ and $d_2 \in S(x_2)$ such that $z_2[x_1, d_1] = z_2[x_2, d_2]$. Since $z_2$ has no neighbours in $V(C)$ by Claim \ref{z2z3vc}, by Claim \ref{s} we have that $|S(z_2)| = 5$, and so there exists a colour $c_3 \in S(z_2)$ and an $i \in \{1,2\}$ such that $c_3 \neq z_2[z_1, c_i]$, such that $x_1[z_2, c_3] \not \in S(x_1)$, and such that $x_2[z_2, c_3] \not \in S(x_2)$: that is, there is a colour choice $c_i$ for $z_1$ that avoids $S(x_1)$, and a colour choice $c_3 \in S(z_2)$ that avoids $c_i \in S(z_1)$ as well as $S(x_1)$ and $S(x_2)$. See Figure \ref{fig:matchings1} for an illustration of the matchings described. Let $C''$ be obtained from $(C \oplus x_1 \oplus x_2) \setminus \{x_1\}$ by adding the vertices $z_1$ and $z_2$ as well as edges $yz_1,z_1z_2,z_2x_2$, where $y \in N(z_1) \cap V(C)$. Let $T'' = (G'', C'', (L,M)) = T \langle C'' \rangle$. Recall that $\deg(x_2) = 6$ by Claim \ref{X2claims} (2); and $N(x_2) \setminus (V(C'') \cup \{x_1\}) = \{z_3, x_3\}$.

We claim that $G'' \setminus \{x_2z_3, x_2x_3\}$ has a $C''$-critical subgraph. To see this, choose $\phi(z_1) = c_i$, and $\phi(z_2) = c_3$. If $\phi$ extends to an $(L,M)$-colouring of $G'' \setminus \{x_1, x_2\}$, then $\phi$ extends to an $(L,M)$-colouring of $G$ by first choosing $\phi(x_2) \in S(x_2) \setminus \{x_2[z_3, \phi(z_3)], x_2[x_3, \phi(x_3)]\}$, and then choosing $\phi(x_1) \in S(x_1) \setminus \{x_1[x_2, \phi(x_2)]\}$. Note this is possible, since $|S(x_1)| = 2$ and $|S(x_2)| = 3$. This is a contradiction, since $\phi$ does not extend to $G$ by assumption. Thus $\phi$ does not extend to an $(L,M)$-colouring of $G''-\{x_1, x_2\}$, and so by Proposition \ref{CriticalSubgraph} we have that $G'' \setminus \{x_2z_3, x_2x_3\}$ has a $C''$-critical subgraph. But then $G''$ contains a proper $C''$-critical subgraph $G^*$. Note that $|E(G'') \setminus E(G^*)| \geq 2$, and by Claim \ref{v(T)}, $v(T'') \geq 3$. Finally, we claim $|E(G^*) \setminus E(C'')| \geq 2$. To see this, note that since $C$ is chordless by Claim \ref{Chord}; since $z_1$ has exactly one neighbour $y$ in $V(C)$ as shown above; since neither $z_2$ nor $z_3$ have a neighbour in $V(C)$ by Claim \ref{z2z3vc}; and since $z_1x_2 \not \in E(G)$ since $G$ is planar, it follows that $C''$ is chordless. Thus $|E(G^*) \setminus E(C'')| \geq 2$. By Claim \ref{ProperCrit} (3) applied to $T''$ and $G^*$, we find that $d(T'') \geq 5-(2\alpha + \varepsilon) -\gamma$. In addition, $s(T_2) \leq s(T'')+2(2\alpha+\varepsilon)$, $\defc(T_2) \geq \defc(T'')-1$, and so $d(T_2) \geq d(T'')-1-2(2\alpha + \varepsilon)$. Thus $d(T_2) \geq 4-\gamma-3(2\alpha+\varepsilon)$. By Claim \ref{QuasiQuasiSame}, $d(T) \geq d(T_2)-2\varepsilon$ and so $d(T) \geq d(T_2)-2\varepsilon \geq 4-\gamma-(6\alpha+5\varepsilon)$. But then $d(T) \geq 3-\gamma$ since by (I2) and (I3) we have that $6\alpha+5\varepsilon \leq 1$. This contradicts the fact that $T$ is a counterexample.

The proof of (ii) is nearly identical. For each $uv \in E(G')$, let $M'_{uv}$ be the restriction of $M_{uv}$ to $S'(u)$ and $S'(v)$. Recall that by Claim \ref{matchmin}, we have that $|M_{z_2x_1}|=5$. Note that $z_2$ is not adjacent to a vertex in $V(C)$ by Claim \ref{z2z3vc}. Thus as $z_2$ is adjacent to $x_1$, we have that $z_2$ has exactly one neighbour in $V(C) \oplus x_1$, and so by Claim \ref{matchmin}, we have that $|S'(z_2)| = 4$. Similarly, $|S'(x_2) = 2|$. Thus there exist two colours $c_1, c_2 \in S'(z_2)$ such that for $i \in \{1,2\}$, we have that $x_2[x_2, c_i] \not \in S'(x_2)$.  Moreover, since $|S'(x_2)|=2$ and $|S'(z_3)|=5$, by Claim \ref{matchmin} we have that that $|M'_{x_2z_3}| = 2$. Recall that $|S(x_3)| = 2$ by Claim \ref{s} and that $|S(z_3)|=5$ by Claims \ref{z2z3vc} and \ref{s}. Since $G$ is planar, neither $x_3$ nor $z_3$ is adjacent to $x_1$. Thus $|M'_{x_3z_2}|= 3$. Suppose for a contradiction that there exist colours $d_2 \in S'(x_2)$ and $d_3 \in S'(x_3)$ such that $z_3[x_2, d_2] = z_3[x_3, d_3]$. Then there exists a colour $c_3 \in S'(z_3)$ and an $i \in \{1,2\}$ such that $c_3 \neq z_3[z_2, c_i]$, such that $x_2[z_3, c_3] \not \in S'(x_2)$, and such that $x_3[z_3, c_3] \not \in S'(x_3)$: that is, there is a colour choice $c_i$ for $z_2$ that avoids $S'(x_2)$, and a colour choice $c_3 \in S'(z_3)$ that avoids $c_i \in S'(z_2)$ as well as $S'(x_2)$ and $S'(x_3)$. See Figure \ref{fig:matchings2} for an illustration of the matchings involved. Let $C'$ be obtained from $(C_1 \oplus x_2 \oplus x_3) \setminus \{x_2\}$ by adding the vertices $z_2$ and $z_3$ as well as edges $x_1z_2,z_2z_3,z_3x_3$. Let $T' = T \langle C'' \rangle = (G', C', (L,M))$. Recall that $\deg(x_3) = 6$ by Claim \ref{degx3}. Let $N(x_3) \setminus (V(C) \cup \{x_2,z_3\}) = \{z_4, z_5\}$.

\begin{figure}[ht]
\tikzset{black/.style={shape=circle,draw=black,fill=black,inner sep=1pt, minimum size=9pt}}
\tikzset{white/.style={shape=circle,draw=black,fill=white,inner sep=1pt, minimum size=9pt}}

\tikzset{invisible/.style={shape=circle,draw=black,fill=black,inner sep=0pt, minimum size=0.1pt}}
\begin{center}
\begin{tikzpicture}
\filldraw[color=black!100, fill=black!0, very thick] (0.5,-3) ellipse (1.1 and 0.7);
\filldraw[color=black!100, fill=black!0, very thick] (7,0) ellipse (0.7 and 1.1);
%\filldraw[color=black!100, fill=black!0, very thick] (-7,0) ellipse (0.7 and 1.7);
\filldraw[color=black!100, fill=black!0, very thick] (-1.7,3) ellipse (1.2 and 0.7);
\filldraw[color=black!100, fill=black!0, very thick] (2,3) ellipse (1.7 and 0.7);

        %\node[] at (-8,0) {$x_1$};
        \node[] at (2.5,-3) {$S'(x_2)$};
        \node[] at (-2.7,3.9) {$S'(z_2)$};
        \node[] at (3,3.9) {$S'(z_3)$};
        \node[] at (7.2,1.5) {$S'(x_3)$};

        \node[white] (1) at (-7,0){};
        \node[white] (2) at (-3.2,3){};
        \node[black] (3) at (-2.6,3){};
        \node[black] (4) at (-2,3){};
        \node[white] (5) at (-1.5,3){$c_1$};
        \node[white] (6) at (-0.9,3){$c_2$};
        \node[white] (7) at (0.8,3){$c_3$};
        \node[black] (8) at (1.4,3){};
        \node[black] (9) at (2,3){};
        \node[black] (10) at (2.6,3){};
        \node[black] (11) at (3.2,3){};
        \node[black] (12) at (7,0.4){};
        \node[white] (13) at (7,-0.4){$d_3$};
        \node[white] (14) at (1,-3){$d_2$};
        \node[black] (15) at (0,-3){};
        \node[white] (16) at (-1,-3){};

        \node[] (17) at (-7.1,0.5){$\phi(x_1)$};
        \node[] at (-4.5,3) {$z_2[x_1, \phi(x_1)]$};
        \node[] at (-2.4,-3) {$x_2[x_1, \phi(x_1)]$};
        \node[] at (-2,0) {$M'_{x_2z_2}$};
        \node[] at (2.5,0) {$M'_{x_2z_3}$};
        \node[] at (4.3,1.3) {$M'_{x_3z_3}$};
        \node[] at (0,2.5) {$M'_{z_3z_3}$};

        \draw[loosely dotted] (1)--(2); 
        \draw[loosely dotted] (1)--(16); 
        \draw[black] (15)--(3); 
        \draw[black] (14)--(4); 
        \draw[black] (12)--(11); 
        \draw[black] (13)--(10); 
        \draw[black] (14)--(10); 
        \draw[black] (15)--(9); 
        \draw[out=90,in=90,dashed]  (5) to  (8);
        \draw[dashed] (6)--(7); 
        %\draw[black] (4)--(5); 
        %\draw[dashed] (2)--(13);
        %\draw[dashed] (1)--(14);

\end{tikzpicture}

\caption{The matchings between $S'(x_2)$, $S'(x_3)$, $S'(z_2)$, and $S'(z_3)$, as described in the proof of the second statement in Claim \ref{keyclaim}. The matching $M'_{x_2x_3}$ is omitted for clarity. We assume there exists a colour $d_2 \in S'(x_2)$ and $d_3 \in S'(x_3)$ such that $z_2[x_2,d_2] = z_2[x_3,d_3]$. Without loss of generality, we may assume that the solid edges in the matchings are as shown. No matter the matching $M'_{z_2z_3}$, there exist colours $c_i \in S'(z_2)$ and $c_3 \in S'(z_2)$ such that $c_i$ is unmatched in $M'_{x_2z_2}$, such that $c_3$ is unmatched in $M'_{x_2z_3}$ and $M'_{x_3z_3}$, and such that $c_3 \neq z_3[z_2, c_i]$.} 
    \label{fig:matchings2}
\end{center}
\end{figure}
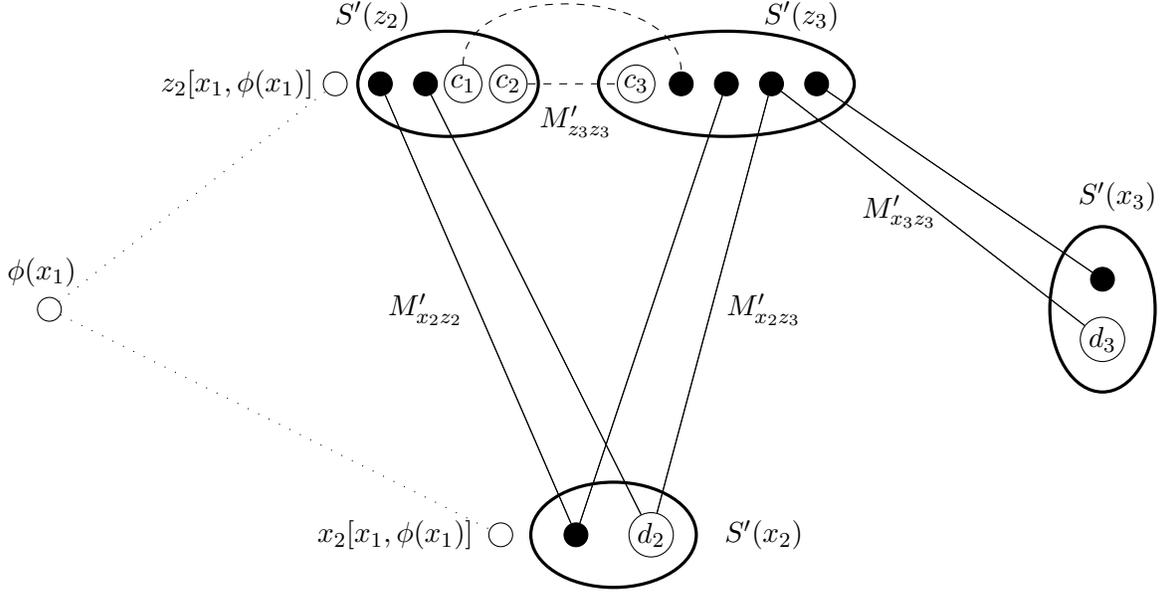

We claim that $G' \setminus \{x_3z_4, x_3z_5\}$ has a $C'$-critical subgraph. To see this, choose $\phi(z_2) = c_i$, and $\phi(z_3) = c_3$. If $\phi$ extends to an $(L,M)$-colouring of $G' \setminus \{x_2, x_3\}$, then $\phi$ extends to an $(L,M)$-colouring of $G$ by first choosing $\phi(x_3) \in S'(x_3) \setminus \{x_3[z_4, \phi(z_4)], x_3[z_5, \phi(z_5)]\}$, and then choosing $\phi(x_2) \in S'(x_2) \setminus \{x_2[x_3, \phi(x_3)]\}$. Note this is possible, since $|S'(x_2)| = 2$ and $|S'(x_3)| = 3$. This is a contradiction, since $\phi$ does not extend to $G$ by assumption. Thus $\phi$ does not extend to an $(L,M)$-colouring of $G'-\{x_2, x_3\}$, and so by Proposition \ref{CriticalSubgraph} we have that $G' \setminus \{x_3z_4, x_3z_5\}$ has a $C'$-critical subgraph. But then $G'$ contains a proper $C'$-critical subgraph $G^*$. Note that $|E(G') \setminus E(G^*)| \geq 2$, and by Claim \ref{v(T)}, $v(T') \geq 3$. Finally, we claim that $C'$ is chordless: this follows from the facts that $C$ is chordless by Claim \ref{Chord}; that neither $z_2$ nor $z_3$ have neighbours in $V(C)$ by Claim \ref{z2z3vc}; and that $z_3z_1 \not \in V(C)$ since $G$ is planar. Thus $|E(G^*) \setminus E(C')| \geq 2$. By Claim \ref{ProperCrit} (3) applied to $T'$ and $G^*$, we find that $d(T') \geq 5-(2\alpha + \varepsilon) -\gamma$.

In addition, $s(T_3) \leq s(T')+2(2\alpha+\varepsilon)$, and $\defc(T_3) \geq \defc(T')-1$. Thus $d(T_3) \geq d(T')-1-2(2\alpha + \varepsilon)$, and so since  $d(T') \geq 5-(2\alpha + \varepsilon) -\gamma$, we have that  $d(T_3) \geq 4-\gamma-3(2\alpha+\varepsilon)$. By Claim \ref{degx3}, $d(T) \geq d(T_3)-3\varepsilon$. Thus $d(T) \geq 4-\gamma-6\alpha-6\varepsilon$. But this is a contradiction, since by (I2) and (I3) we have that $6\alpha+6\varepsilon \leq 1$.
\end{proof}

We now show $x_3$ is not of type (2,0,0), thus contradicting Claim \ref{x3type} and completing the proof of Theorem \ref{theorem:stronglinear}.

\begin{claim}\label{not200}
$x_3$ is not of type (2,0,0).
\end{claim}
\begin{proof}
Suppose not. By Claim \ref{X1claims} (1), we have that $x_1$ is a tripod of $T$. By Claim \ref{s}, it follows that $|S(x_1)| = 2$. Let $|S(x_1)| = \{c_1, c_2\}$, and for each $i \in \{1,2\}$, let $\phi_i$ be an extension of $\phi$ to $x_1$ with $\phi_i(x_1) = c_i$; let $S_i(x_2) = S(x_2) \setminus x_2[x_1, c_i]$; and similarly let $S_i(x_3) = S(x_3) \setminus x_3[x_1, c_i]$. Note that $S_1(x_2) \neq S_2(x_2)$ since $c_1$ and $c_2$ are distinct. Furthermore, note that $x_3$ is not adjacent to $x_1$ since $x_3$ is a tripod of $T_2$ of type (2,0,0): thus $S_i(x_3) = S(x_3)$ is a fixed set that does not depend on $i$. Let $M^i_{z_3x_2}$ be the restriction of $M_{z_3x_2} $ to $S_i(z_3)$ and $S_i(x_2)$.  Finally, let $S = S(z_3)\setminus \{z_3[x_3, d] : d \in S_i(x_3)\}$.  Since $S_i(x_3)$ is fixed, so too is $S$.

Note that by Claim \ref{matchmin}, we have that $|M^i_{z_3x_2}| = 2$ for each $i \in \{1,2\}$, and moreover by Claim \ref{keyclaim} (2) there does not exist a colour $d_3 \in S_i(x_3)$ and a colour $d_2 \in S_i(x_2)$ such that $z_3[x_2, d_2] = z_3[x_3, d_3]$.  Since $S$ is fixed, this implies that $M^1(z_3x_2) = M^2(z_3x_2)$. This is a contradiction, since $S_1(x_2)$ and $S_2(x_2)$ are distinct sets of size two.
\end{proof}

%======================================================================
\section{Implications}\label{sec:implications}
%======================================================================

In this section, we prove Theorem \ref{theorem:linearcycle}, and discuss the implications of this result. %We will use the following equivalent definition of deficiency.

%\begin{lemma}[Lemma 3.3, \cite{postle2016five}]\label{alternatedef}
%If $G$ is a 2-connected plane graph with outer cycle $C$, then 
%$$
%\defc(G) = |V(C)|-3-\sum_{f \in \mathcal{F}(G)}(|f|-3).
%$$
%\end{lemma}

%Theorem \ref{theorem:stronglinear} implies the following.
%\begin{thm}\label{6.1equiv}
%If $T=(G,C,(L,M))$ is a critical canvas, then 
%$$
%\varepsilon |V(G)\setminus V(C)| + \sum_{f \in \mathcal{F}(G)}(|f|-3) \leq |V(C)|-4.
%$$
%\end{thm}
%\begin{proof}
%If $T= (G,C, (L,M))$ is a critical canvas, then it follows from Lemma \ref{alternatedef} that $d(T) = %|V(C)|-3-\sum_{f \in \mathcal{F}(G)}(|f|-3)-\varepsilon v(G)-\alpha(b(T)+q(T))$. Thus $d(T) \leq |V(C)|-3-\sum_{f \in \mathcal{F}(G)}(|f|-3)-\varepsilon v(G)$.  By Theorem \ref{theorem:stronglinear}, if $v(T) \geq 2$ then $3-\gamma \leq d(T)$\textcolor{blue}{; and $3- \gamma \geq 1$ by (I3) in Theorem \ref{theorem:stronglinear}}. If $v(T) \leq 1$, then \textcolor{blue}{it follows immediately from the definition of $d(\cdot)$ that} $d(T) \geq 1$. Thus $1 \leq |V(C)|-3-\sum_{f \in \mathcal{F}(G)}(|f|-3)-\varepsilon v(G)$, and so 
%$$\varepsilon |V(G)\setminus V(C)| + \sum_{f \in \mathcal{F}(G)}(|f|-3) \leq |V(C)|-4,$$
%as desired.
%\end{proof}
%Omitting the \textcolor{blue}{sum over the faces} from Theorem \ref{6.1equiv} gives the theorem below.
\begin{thm}\label{5CCconstant}
If $T = (G,C,(L,M))$ is a critical canvas and $\varepsilon$ is as in Theorem \ref{theorem:stronglinear}, then $|V(G)| \leq \frac{1+\varepsilon}{\varepsilon}|V(C)|-\frac{4}{\varepsilon}$.
\end{thm}
\begin{proof}
%By Theorem \ref{6.1equiv}, \textcolor{blue}{$|V(G)\setminus V(C)| \leq \frac{1}{\varepsilon}|V(C)|$. The result follows.} %Thus $|V(G)| \leq \frac{1}{\varepsilon}|V(C)| + |V(C)|$, and so the result follows.

Let $F(G)$ be the set of faces in the embedding of $G$. Note that
\begin{align*}
    |E(G)| - 3|V(G)| &= |E(G)| - 3(|E(G)| + 2 - |F(G)|) \textnormal{ \hskip 4mm by Euler's formula} \\
 &= 3|F(G)| - 2|E(G)| - 6  \\
&= \sum_{f \in F(G)} (3-|f|) - 6 \\
&\leq (3-|V(C)|) - 6  \textnormal{\hskip 4mm since each $f \in F(G)$ has degree at least 3} \\
&= -|V(C)|-3.
\end{align*}
Moreover, by definition, $\defc(G) = |E(G)|-|E(C)|-3|V(G)|+3|V(C)|= |E(G)|-3|V(G)|+2|V(C)|$. Using the inequality above, it follows that $\defc(G) \leq (-|V(C)|-3)+2|V(C)$, or $\defc(G) \leq |V(C)|-3$.
%\textcolor{blue}{Note that by definition, $\defc(G) = |E(G)|-|E(C)|-3|V(G)|+3|V(C)|= |E(G)|-3|V(G)|+2|V(C)|$.}
%\textcolor{blue}{Let $\mathcal{F}(G)$ be the set of internal faces of $G$. By Euler's formula, $3|\mathcal{F}(G)| - 3 =3|E(G)| - 3|V(G)|$; or, adding $2|V(C)|$ to each side and rearranging, $|V(C)|-3+ 3|\mathcal{F}|-(2|E(G)|-|V(C)|) = |E(G)|-3|V(G)|+2|V(C)|$, which is precisely equal to our expression for $\defc(G)$, above. Moreover, $\sum_{f \in \mathcal{F}(G)}(|f|-3) = 2 |E(G)| -|V(C)| - 3 |\mathcal{F}(G)|$. Thus
%\begin{align*}
%    \defc(G) &= |V(C)|-3 + 3|\mathcal{F}|-(2|E(G)|-|V(C)|) \\
%    &= |V(C)| -3 - \sum_{f \in \mathcal{F}} (|f|-3) \\
%    &\leq |V(C)| - 3 \textnormal{ \hskip 3mm since internal faces have degree at least 3.}
%\end{align*}}
Now, $d(T) \leq \defc(G)-\varepsilon (|V(G)| - |V(C)|)$ by definition, and so since $\defc(G) \leq |V(C)|-3$, we have that $d(T) \leq |V(C)|-3-\varepsilon(|V(G)| - |V(C)|)$. By Theorem \ref{theorem:stronglinear}, if $v(T) \geq 2$ then $3-\gamma \leq d(T)$; and $3- \gamma \geq 1$ by (I3) in Theorem \ref{theorem:stronglinear}. If $v(T) \leq 1$, then it follows immediately from the definition of $d(\cdot)$ and the fact that internal vertices of $G$ have degree at least 5 (see Proposition \ref{Facts}) that $d(T) \geq 1$. Thus $1 \leq d(T) \leq |V(C)|-3-\varepsilon(|V(G)|-|V(C)|)$, or $4 \leq |V(C)|(1 + \varepsilon) - \varepsilon |V(G)|$. The result follows by rearranging.
\end{proof}

To obtain the best possible bound in Theorem \ref{5CCconstant}, we wish to maximize $\varepsilon$ subject to inequalities (I1-I3). To that end, we choose  $\alpha = \frac{1}{25}$, $\varepsilon = \frac{1}{50}$, and $\gamma = \frac{7}{10}$, giving $V(G) \leq 51 |V(C)| $ in Theorem \ref{5CCconstant}.

Theorem \ref{theorem:linearcycle} follows from Theorem \ref{5CCconstant} as shown below.
\begin{proof}[Proof of Theorem \ref{theorem:linearcycle}]
Let $G, C, (L,M),$ and $H$ be as in Theorem \ref{theorem:linearcycle}. We claim that $H$ is $C$-critical. Suppose not. Then there exists a proper subgraph $H'$ of $H$ such that every $(L,M)$-colouring of $C$ that extends to $H'$ also extends to $H$. But since every $(L,M)$-colouring $C$ that extends to $H$ also extends to $G$, we have that $H'$ contradicts the minimality of $H$. Thus $H$ is $C$-critical, and so by Theorem \ref{5CCconstant} we have $|V(H)| \leq 51 |V(C)|$, as desired.
\end{proof}

We next show how Theorem \ref{5CCconstant} implies that the family of embedded graphs that are critical for 5-correspondence colouring forms a hyperbolic family. Note the theorem below is merely a more explicit version of Theorem \ref{5cchyperbolic}. Following this, we discuss the implications of the hyperbolicity of a family of graphs as described by Postle and Thomas in \cite{postle2018hyperbolic}.

\begin{thm}\label{hyperbolic5cc}
The family $\mathcal{F}$ of embedded graphs that are critical for 5-correspondence colouring is hyperbolic with Cheeger constant $50$.
\end{thm}
\begin{proof}
Let $G$ be a graph that is $(L,M)$-critical, where $(L,M)$ is a 5-correspondence colouring. Note that $G$ is connected, as otherwise since every subgraph of $G$ admits an $(L,M)$-colouring, it follows that each component of $G$ admits an $(L,M)$-colouring and thus so does $G$ itself, contradicting the definition of $(L,M)$-critical. Suppose that $G$ is embedded in a surface $\Sigma$, and let $\lambda: \mathbf{S}^1 \rightarrow \Sigma$ be a closed curve intersecting $G$ in only its vertices and bounding an open disk $\Delta$. Let $Y$ be the set of vertices of $G$ that are intersected by $\lambda$, and let $X$ be the set of vertices in $\Delta$. The theorem follows from showing that if $X$ is non-empty, then  $|X| \leq 50(|Y|-1)$.  Let $G_1 : = G[X \cup Y]$, and let $G_2 := G\setminus G_1$. Since $G$ is critical for 5-correspondence colouring, there exists a colouring of $G_2$ that extends every proper subgraph of $G$ containing $G_2$ but not to $G$ itself. Since $X  \neq \emptyset$, it follows that $G_1$ is $G[Y]$-critical. By Theorem \ref{tech5CC}, it follows that $|Y| \geq 3$. Let $v_0,v_1,v_2, \cdots, v_k$ be the vertices of $Y$ appearing in a cyclic order along $\lambda$. Let $C$ be the cycle $v_0v_1\cdots v_kv_0$. Since $G_1$ is $G[Y]$-critical, it follows that $G_1 \cup C$ is $C$-critical. By Lemma \ref{2connhyp}, $G_1 \cup C$ is 2-connected, and hence $(G_1, C, (L,M))$ is a canvas. By Theorem \ref{5CCconstant} with $\varepsilon = \frac{1}{50}$, we have that $|V(G_1)\setminus V(C)| \leq  50 (|V(C)|-1)$. The result follows.
\end{proof}

Showing that such a family of critical graphs is hyperbolic has many interesting implications, as described in \cite{postle2018hyperbolic}. We highlight a few in particular below, following a definition.

\begin{definition}
A \emph{non-contractible cycle} in a surface is a cycle that cannot be continuously deformed to a single point. An embedded graph is \emph{$\rho$-locally planar} if every cycle (in the graph) that is non-contractible (in the surface) has length at least $\rho$.
\end{definition} 

In \cite{postle2018hyperbolic}, Postle and Thomas show the following.

\begin{thm}[Postle \& Thomas, \cite{postle2018hyperbolic}]\label{lukelocplthm}
For every hyperbolic family $\mathcal{F}$ of embedded graphs that is closed under curve cutting there exists a constant $k > 0$ such that every graph $G \in \mathcal{F}$ embedded in a surface of Euler genus $g$ has a non-contractible cycle of length at most $k\log(g + 1)$.
\end{thm}

Using this, Theorem \ref{locallyplanar} follows as a corollary to Theorem \ref{hyperbolic5cc}.  In addition, following the work of Dvo{\v{r}}{\'a}k and Kawarabayashi in \cite{dvovrak2013list}, Theorem \ref{theorem:linearcycle} implies Theorems \ref{alg1} and \ref{alg2}. Note that by \emph{linear-time algorithms}, we mean algorithms whose run-time is linear in the number of vertices in the graph. 

The algorithms in the theorems above are the same as those given by Dvo{\v{r}}{\'a}k and Kawarabayashi in \cite{dvovrak2013list}. We refer to \cite{dvovrak2013list} for a complete description of the algorithms and the proof of correctness. An overview of the algorithms is given in Chapter 3, Section 6 of \cite{evethesis}. 
\section{The Girth At Least Five Case}\label{sec:ggeq5}
%======================================================================
In this section, we redefine \emph{canvas} as follows.
\begin{definition}
We say the triple $(G,S,(L,M))$ is a \emph{canvas} if $G$ is a plane graph, $S$ is any connected subgraph of $G$, and $(L,M)$ is a correspondence assignment for $G$ such that there exists an $(L,M)$-colouring of $S$ and $G$ has girth at least five and $|L(v)| \geq 3$ for all $v \in V(G) \setminus V(S)$.
\end{definition}
The main result of this section is the following observation.
\begin{obs}\label{girth5:stronglinear}
Let $\varepsilon,\alpha >0$ satisfy the following:  $9\varepsilon \leq \alpha$; $2.5 \alpha + 5.5 \varepsilon \leq 1$; and $11 \varepsilon + 1 \leq 3\alpha$.
If $T = (G, S, L)$ is a critical canvas where:
\begin{itemize}[itemsep=0pt]
    \item $G$ has girth at least five,
    \item $G$ is not composed of exactly $S$ and one edge not in $S$,
    \item $G$ is not composed of exactly $S$ together with one vertex of degree $3$, then
\end{itemize}  then $3e(T)- (5 + \varepsilon)v(T)-\alpha q(T) \geq 3$. \end{obs}

This is the correspondence colouring analogue of a nearly identical theorem for list colouring of Postle: Theorem 3.9, \cite{postle20213}. Beyond the change from list colouring to correspondence colouring, the key difference between the statements of Theorem \ref{girth5:stronglinear} and Theorem 3.9 in \cite{postle20213} is that $S$ is connected (as opposed to having at most two components). This change allows us to use Theorem 2.11 in \cite{postle20213} (which describes structures arising from critical canvases $(G,S,(L,M))$ where $S$ is connected, and which holds for correspondence colouring) in lieu of Theorem 2.12 (which allows $S$ to have two components, and which is not currently known to hold for correspondence colouring). Otherwise, the proof of Theorem 3.9 in \cite{postle20213} carries over to the correspondence colouring framework with only standard, minor changes: namely, when we perform reductions (colouring a strict subgraph of a minimum counterexample, deleting this subgraph, and removing vertices' colours from neighbours' lists), we delete \emph{corresponding} colours from neighbouring lists, rather than identical colours.

The proof is similar in spirit to that of Theorem \ref{theorem:stronglinear}. However, as noted in Section \ref{sec:challenges}, Postle and Thomas' list colouring theorem in the 5-choosability case does \emph{not} carry over to correspondence colouring. The colouring arguments in Postle and Thomas' theorem for 5-choosability rely on the fact that for a triangle $ux_2z_2u$ in a minimum counterexample with list assignment $S$, if $S(u) \subseteq S(x_2)$, then $S(z_2) \setminus (S(x_2) \cup S(u)) = S(z_2) \setminus S(x_2)$. This implies that it is possible to colour $z_2$ from $S(z_2)$ while avoiding the lists of both $x_2$ and $u$. This argument crucially does not hold for correspondence colouring: an analogous argument to that in shows merely that for a correspondence assignment $(S,M)$, we have $|M_{x_2u}| = |S(u)|$, which of course implies nothing about $M_{z_2u}$. Crucially, the proof in the 5-choosability case involves keeping track of lists along a cycle. This is not the case in Postle's proof for 3-choosability: the colouring arguments involve only deleting vertices and removing their colours (or in the correspondence framework, their \emph{corresponding} colours) from the lists of neighbours, and do not keep track of what these colours correspond to. Moreover, no arguments rely on keeping track of what colours are or are not available in a cycle: the colouring arguments involve only trees branching from vertices in $S$ in the minimum counterexample.

Theorem \ref{girth5:stronglinear} implies the following (the correspondence colouring analogue of Theorem 1.8 in \cite{postle20213}).
\begin{obs}\label{girth5:3ccbound} 
Let $G$ be a plane graph of girth at least five, let $(L,M)$ be a 3-correspondence assignment for
$G$, and let $C$ be a facial cycle of $G$. If $G$ is $C$-critical with respect to $(L,M)$, then $|V(G)| \leq 89|V(C)|$.
\end{obs}

Observation \ref{girth5:3ccbound} in turn implies an important corollary below. First, we will need the following theorem, due to Thomassen. This theorem was originally written in the language of list colouring; however, as pointed out by Dvo{\v{r}}{\'a}k and Postle in \cite{dvovrak2018correspondence}, the proof also carries over to the realm of correspondence colouring.
\begin{thm}[Thomassen, \cite{thomassen3LCnew}]\label{thomtech3cc}
Let $G$ be a plane graph of girth at least five. Let $C$ be the subgraph of $G$ whose edge- and vertex-set are precisely those of the outer face boundary walk of $G$. Let $(L,M)$ be a correspondence assignment for $G$ where $|L(v)| \geq 1$ for each vertex $v$ in a path or cycle $S \subseteq C$ with $|V(S)|\leq 6$; where $|L(v)| = 2$ for each vertex $v$ in an independent set $A$ of vertices in $|V(C) \setminus V(S)|$; where $|L(v)|\geq 3$ for all $v \in V(G) \setminus (A \cup V(S))$; and where there is no edge between vertices in $A$ and vertices in $S$. Then every $(L,M)$-colouring of $S$ extends to an $(L,M)$-colouring of $G$.
\end{thm}

The following corollary follows from Observation \ref{girth5:3ccbound}. The proof is very similar to that of Theorem \ref{hyperbolic5cc}, but uses Theorem \ref{thomtech3cc} instead of Theorem \ref{tech5CC}. See Lemma 5.13 in \cite{postle2018hyperbolic} for a proof of the list colouring case, which is nearly identical.
\begin{cor}\label{girth5:hyp}
The embedded graphs of girth at least five that are critical for 3-correspondence colouring form a hyperbolic family.
\end{cor}

Similar to the girth 3 case, Corollary  \ref{girth5:hyp} implies Theorem \ref{locallyplanar5}, and Observation \ref{girth5:3ccbound} implies Theorems \ref{alg15} and \ref{alg25}.

%======================================================================
%Acknowledgement
%======================================================================
\noindent
\paragraph{Acknowledgement.}
The results in this paper form part of the doctoral dissertation \cite{evethesis} of the second author, written under the guidance of the first.

\bibliographystyle{siam}
\bibliography{bibliog}
\end{document}